\documentclass[a4paper,reqno]{amsart}

\usepackage[czech,english]{babel}
\usepackage{eufrak}
\usepackage{graphicx}
\usepackage{amssymb}
\usepackage{amscd}
\usepackage{amsthm}
\usepackage{amsmath}
\usepackage{enumerate}
\usepackage{mathrsfs}
\usepackage{graphics}
\usepackage[all,cmtip]{xy}
\usepackage{tikz}
\usepackage{extarrows}
\usepackage{tikz-cd}
\usepackage{ulem}
\usetikzlibrary{calc}
\usetikzlibrary{matrix,arrows,decorations.pathmorphing}
\usetikzlibrary{arrows,decorations.markings, matrix}

\usepackage[colorlinks=true,pagebackref,hyperindex]{hyperref}
\newcommand\myshade{85}
\colorlet{mylinkcolor}{violet}
\colorlet{mycitecolor}{red}
\colorlet{myurlcolor}{cyan}

\hypersetup{
  linkcolor  = mylinkcolor!\myshade!black,
  citecolor  = mycitecolor!\myshade!black,
  urlcolor   = myurlcolor!\myshade!black,
  colorlinks = true,
}

\newtheorem{Thm}{Theorem}[section]\newtheorem*{Thm*}{Theorem}
\newtheorem{theoremA}{Theorem}

\newtheorem{Lem}[Thm]{Lemma}
\newtheorem{Cor}[Thm]{Corollary}

\newtheorem{Prop}[Thm]{Proposition}
\newtheorem{Prop-Def}[Thm]{Proposition-Definition}
\newtheorem*{Conj*}{Conjecture}

\theoremstyle{definition}
\newtheorem{Ex}[Thm]{Example}
\newtheorem{Def}[Thm]{Definition}

\newtheorem{Rem}[Thm]{Remark}

\newcommand{\w}{\widetilde}
\newcommand{\ra}{\rightarrow}
\newcommand{\la}{\leftarrow}
\newcommand{\D}{\mathcal{D}}
\newcommand{\C}{\mathcal{C}}
\newcommand{\K}{\mathsf{K}}

\newcommand{\cal}{\mathcal}
\newcommand{\T}{\mathcal T}
\newcommand{\X}{\mathcal X}
\newcommand{\Y}{\mathcal Y}
\newcommand{\ZZ}{\mathcal Z}

\newcommand{\f}{\mathrm{f}}
\newcommand{\Tot}{\mathrm{Tot}}
\newcommand{\bb}{\mathrm{b}}

\newcommand{\Loc}{\mathsf{Loc}}
\newcommand{\Filt}{\mathsf{Filt}}

\newcommand{\h}{{\mathrm H}}
\newcommand{\U}{{\mathcal U}}
\renewcommand{\H}{{\mathcal H}}
\renewcommand{\dim}{{\rm dim}}

\newcommand{\xra}{\xrightarrow}
\newcommand{\Z}{{\mathbb Z}}

\newcommand{\lan}{\langle}
\newcommand{\ran}{\rangle}

\newcommand{\op}{\oplus}
\newcommand{\bop}{\bigoplus}
\newcommand{\ot}{\otimes}

\newcommand{\hs}{\hspace{-3pt}}

\newcommand{\Hom}{\operatorname{Hom}\nolimits}
\newcommand{\Top}{\operatorname{top}\nolimits}

\newcommand{\End}{\operatorname{End}\nolimits}
\newcommand{\RHom}{\mathbf{R}\strut\kern-.2em\operatorname{Hom}\nolimits}
\newcommand{\RshHom}{\mathbf{R}\strut\kern-.2em\mathscr{H}\strut\kern-.3em\operatorname{om}\nolimits}
\newcommand{\shHom}{\mathscr{H}\strut\kern-.3em\operatorname{om}\nolimits}
\newcommand{\shEnd}{\mathscr{E}\strut\kern-.3em\operatorname{nd}\nolimits}

\newcommand{\za}{\alpha}
\newcommand{\zb}{\beta}
\newcommand{\zd}{\delta}
\newcommand{\zD}{\Delta}
\newcommand{\ze}{\epsilon}
\newcommand{\zg}{\gamma}
\newcommand{\zG}{\Gamma}

\DeclareMathOperator{\moduleCategory}{{\mathsf{mod}}} \renewcommand{\mod}{\moduleCategory}
\DeclareMathOperator{\proj}{\mathsf {proj}}
\DeclareMathOperator{\thick}{\mathsf{thick}}
\DeclareMathOperator{\per}{\mathsf{per}}
\DeclareMathOperator{\add}{\mathsf {add}}

\numberwithin{equation}{section}

\definecolor{dark-green}{RGB}{14,150,2}
\definecolor{red}{RGB}{250,0,0}
\newcommand{\gpoint}{\color{dark-green}{\circ}}
\newcommand{\rpoint}{\color{red}{\bullet}}

\begin{document}
\normalem
\title[Recollements of partially wrapped Fukaya categories]{Recollements of partially wrapped Fukaya categories and surface cuts}

\author{Wen Chang}
\address{School of Mathematics and Statistics, Shaanxi Normal University, Xi'an 710062, China}
\email{changwen161@163.com}

\author{Haibo Jin}
\address{Department of Mathematics, Universit\"at zu K\"oln, Weyertal 86-90, 50931 K\"oln, Germany}
\email{hjin@math.uni-koeln.de}

\author{Sibylle Schroll}
\address{Department of Mathematics, Universit\"at zu K\"oln, Weyertal 86-90, 50931 K\"oln, Germany and
Institutt for matematiske fag, NTNU, N-7491 Trondheim, Norway}
\email{schroll@math.uni-koeln.de}



\date{\today}


\begin{abstract}

In this paper we use  recollements to investigate  partially wrapped Fukaya categories of surfaces with marked points. In particular, we show that cutting surfaces gives rise to  recollements of the corresponding partially wrapped Fukaya categories.
Our approach is based on the fact that the partially wrapped Fukaya category of a surface with marked points is triangle equivalent to the perfect derived category of a homologically smooth and proper graded gentle algebra with zero differential as shown by Haiden, Katzarkov and Kontsevich.
Using this, we  study particular generators of partially wrapped Fukaya categories, namely full exceptional sequences, silting objects and simple-minded collections. In particular, we fully characterise the existence of full exceptional sequences and we give an example of a partially wrapped Fukaya category which does not admit a silting object, that is a generator with no positive self-extensions.
\end{abstract}

\maketitle


\section*{Introduction}\label{Introductions}

Gentle algebras, first arising in the representation theory of associative algebras \cite{AH81,AS87}, connect to many other areas of mathematics such as Lie algebras, dimer models, cluster theory, N=2 Gauge theories and homological mirror symmetry. In the context of homological mirror symmetry, the partially wrapped Fukaya category of a graded marked surface is triangle equivalent to the bounded derived category of a homologically smooth graded gentle algebra with zero differential \cite{HKK17, LP20}.

On the other hand, the notion of recollements of triangulated categories was first introduced by A.~A.~Beilinson, J.~Bernstein and P.~Deligne in \cite{BBD} (see Definition \ref{recollement} for more details).
 Recollements play an important  role in algebraic geometry and representation theory, in particular when comparing derived categories and their invariants (see for example \cite{AKLY,H,Han,Keller98,Koenig91,KoenigNagase}).

In this paper, we use recollements to investigate partially wrapped Fukaya categories of marked surfaces. In particular, we show that cutting the surface gives rise to a recollement of the associated partially wrapped Fukaya categories. More precisely, we show

\begin{theoremA}[Theorem \ref{Thm:RecollementFukaya}]
Let  $\mathcal{W}(S,M,\eta(\zD))$ be the partially wrapped Fukaya category associated to a graded surface dissection $(S, M, \zD)$  with induced line field $\eta(\zD)$.  Let $L \subset \zD$ be a subset of $\zD$ such that  the graded marked surface $\mathcal{W}(S(L),M(L),\eta(L))$ generated by $L$ (as graded ribbon subgraph of $\zD$) has no marked points in the interior. Then there is a recollement of partially wrapped Fukaya categories

\begin{equation}\label{eq:recollementFukayaA}
\xymatrix @C=3pc{
\mathcal{W}(S_L,M_L,\eta(\zD_L)) \ar[r]^{}
& \mathcal{W}(S,M,\eta(\zD)) \ar@/_1.5pc/[l]_{} \ar@/^1.5pc/[l]^{}
 \ar[r]^{}
&\mathcal{W}(S(L),M(L),\eta(L))  \ar@/^-1.5pc/[l]_{} \ar@/^1.5pc/[l]^{}
  } \end{equation}
where the left side corresponds to the graded cut marked surface along $L$.
\end{theoremA}

Our proof  is based on the fact that the partially wrapped Fukaya category $\mathcal(S,M, \eta(\zD))$ is triangle equivalent to the bounded derived category of a proper and homologically smooth differential graded gentle algebra with zero differential \cite{HKK17, LP20}.
In fact, we start with a more general setting. Namely,  we first consider recollements arising from graded quadratic monomial algebras.   For this, given a graded quadratic monomial algebra $A=kQ/I$, we introduce the notion of \emph{partial cofibrant dg algebra resolution} $A_{J}$ with respect to a subset $J$ of $I$ by `resolving' only the relations in $J$ (see Definition \ref{Def:AJ} for the details). Then we show the following
\begin{theoremA}[Theorem \ref{Thm:AJA}]
Let $A=kQ/I$ be a graded quadratic monomial algebra and let $J\subset I$. Let $A_{J}$ be the partial cofibrant dg algebra resolution of $A$. Then
there is a natural quasi-isomorphism $A_{J}\simeq A$ of dg $k$-algebras.
\end{theoremA}
Moreover, if $J$ is the set of relations at some vertices corresponding to an idempotent $e\in A$, the above quasi-isomorphism induces the following  recollement (see Theorem~\ref{Thm:recollement1})
\begin{equation}\label{IntroRecoll}
\xymatrixcolsep{4pc}\xymatrix{
\D(A_{e}) \ar[r]^{i_{*}=i_{!}}
&\D(A) \ar@/_1.5pc/[l]_{i^{*}} \ar@/^1.5pc/[l]^{i^{!}}
 \ar[r]^{?\ot_{A}^{\bf L}Ae}
&\D(eAe) \ar@/^-1.5pc/[l]_{?\ot^{\bf L}_{k}eA} \ar@/^1.5pc/[l]^{\RshHom_{eAe}(Ae, ?)}
  },\end{equation}
where $A_e$ is again a quadratic monomial algebra explicitly described in terms of quiver and relations defined in Definition~\ref{definition:left algebra}.
We remark that there are serval different ways to construct  dg algebras $B$ such that $\D(B)$ fits the left side of the recollement in \eqref{IntroRecoll}, see for example \cite{BCL, Drinfeld,KY, Ke99, NS}. Here we give a new and concrete construction for any graded quadratic monomial algebra, which plays an important role in our study of partially wrapped Fukaya categories.

If in the recollement in \eqref{IntroRecoll}  two out of the three algebras $A, eAe,$ and  $A_e$ are homologically smooth and proper, so is the third one. Moreover, in this case \eqref{IntroRecoll} restricts to a recollement of perfect derived categories  (see Proposition \ref{Thm:Fukaya}) and gives rise to the recollement \eqref{eq:recollementFukayaA}. More generally we show that it can  be extended to an unbounded ladder of recollements (see Proposition~\ref{Prop:ladder})  as introduced  in \cite{AKLY}.

The second aim of this paper is to study particular generators of partially wrapped Fukaya categories, namely, full exceptional sequences,  silting objects and  simple minded collections. Exceptional sequences originate in algebraic geometry \cite{GR, Ru} and are instrumental in the construction of the Fukaya-Seidel categories \cite{Seidel}.
Silting objects and simple-minded collections are some of the standard tools for studying triangulated categories and are   important  in Koszul duality
\cite{BGS, Keller94,KN, KoY,R}.

We first note that in terms of the associated homologically smooth, proper graded gentle algebras, the recollement in  \eqref{IntroRecoll} gives rise to reduction with respect to silting objects  of the associated  perfect derived categories. This induces  a bijection (see Theorem \ref{Thm:siltreduc})
\[ \{\mbox{silting objects in $\per(A_{e})$}\} \overset{1:1}{\longleftrightarrow} \{ \mbox{silting objects in $\per(A)$ containing $P$}\},\]
which gives us a way to study silting theory of $A$ in terms of the `smaller' algebra $A_{e}$.
We show that a similar  result holds for reduction of simple-minded collections.

We then study which partially wrapped Fukaya categories admit full exceptional sequences,  silting objects and simple-minded collections.

In general, it is difficult to determine whether a triangulated  category has such generators or not.  In the case of a partial wrapped Fukaya category,
 we first give a full characterization of  the existence of full exceptional sequences. 
 \begin{theoremA}[Theorem \ref{Prop:excep}]
Let $\mathcal{W}(S,M,\eta(\zD))$ be a partially wrapped Fukaya category.
 Then $\mathcal{W}(S,M,\eta(\zD))$ has a full exceptional sequence if and only if
 $(S,M)$ is not homeomorphic to a marked surface of genus $\ge 1$ with exactly one boundary and one marked point on that boundary.
 \end{theoremA}

 Next we study the existence of silting objects and simple-minded collections.
Denote by $A^{(n)}$ the family of graded gentle algebras with parameters $a_{i}, b_{i}$ for $1\le i\le n$ associated to a surface of genus $n$ with exactly one boundary component and one marked point on that boundary component (see  Section~\ref{Sec:Exceptional} for the precise definitions). We then show the following.

\begin{theoremA}[Theorem \ref{Thm:exist}]\label{Thm:existA}
Let $\mathcal{W}(S,M,\eta(\zD))$ be a partially wrapped Fukaya category.
 Then $\mathcal{W}(S,M,\eta(\zD))$ has   a  silting object  if and only if it has a simple-minded collection and this is the case if
\begin{enumerate}[\rm (1)]
\item $\mathcal{W}(S,M,\eta(\zD))$ is not triangle equivalent to $\D^{\rm b}(A^{(n)})$, or
\item $\mathcal{W}(S,M,\eta(\zD))$ is  triangle equivalent to $\D^{\rm b}(A^{(n)})$ with $a_{i}\not=1$ or $b_{i}\not=1$ for any $1\le i\le n$.
\end{enumerate}
\end{theoremA}

We note that in case that $n=1$, the converse of Theorem  \ref{Thm:existA} holds. Namely, 
we show by hand that $A^{(1)}$ with the grading $a_1=b_1=1$ does not admit a silting object by checking in a case by case analysis, that every  indecomposable object in $\D^{\rm b}(A^{(1)})$ has a positive self-extension.

\subsection*{Acknowledgements} The authors would like to thank Bernhard Keller and Dong Yang for helpful comments regarding recollements and cofibrant dg algebra resolutions.


\section*{Conventions}\label{Conventions}
In this paper, all algebras will be assumed to be over a base field $k$.
Arrows in a quiver are composed from left to right as follows: for arrows $a$ and $b$ we write $ab$ for the path from the source of $a$ to the target of $b$.
We adopt the convention that maps are composed from right to left, that is if $f: X \to Y$ and $g: Y \to Z$ then $gf : X \to Z$. In general, we consider right modules.
We denote by $\mathbb{Z}$ the set of integer numbers, and by $\mathbb{Z}^*$ the set of non-zero integer numbers. We use notation $|M|$ to represent the number of elements in a finite set $M$.

\section{Preliminaries}\label{Preliminaries}

\subsection{DG algebras}
Let $A$  be a differential graded (dg) $k$-algebra, that is, a graded algebra endowed with a compatible structure of a $k$-complex. A \emph{(right) dg $A$-module} is a graded $A$-module endowed with a compatible structure of a $k$-complex. Let $\D(A)$ be the derived category of right dg $A$-modules (see \cite{Keller94, Keller06}). It is a triangulated category obtained from the category of dg $A$-modules by formally inverting all quasi-isomorphisms. The shift functor is given by the shift of complexes.

We denote by $\per(A)$ the \emph{perfect category} of $A$, which is the smallest triangulated subcategory of $\D(A)$ generated by $A$ and which is closed under direct summands. We denote by $\D^{\bb}(A)$ the full subcategory of $\D(A)$ consisting of the objects $M$ whose total cohomology is finite-dimensional (that is, $\dim_{k}\bop_{n}\h^{n}(M)<\infty$).  If $\per (A)\subset \D^{\bb}(A)$, that is, the cohomology algebra $\H^{*}(A)$ is finite dimensional, we say $A$ is \emph{proper}. If $A\in \per (A\ot_{k}A^{\rm op})$, we say $A$ is \emph{homologically smooth}.
The following observations are well-known.
\begin{Lem}\label{Lem:Db=per}
Let $A$ be a homologically smooth dg $k$-algebra. Then $\D^{\bb}(A)\subset \per (A)$. If $A$ is also proper, then $\D^{\bb}(A)=\per (A)$.
\end{Lem}

\begin{Lem} \label{Lem:totalcoh}
Let $A$ be a dg $k$-algebra and let $X$ be a dg $A$-module. Then the following are equivalent.
\begin{enumerate}[\rm(1)]
\item $X\in \D^{\bb}(A)$.
\item For any $P\in \per(A)$, $\RshHom_{A}(P, X)$ has finite dimensional total cohomology.
\end{enumerate}
\end{Lem}

\begin{Lem}\cite{Keller94} \label{Lem:D|per}
Let $A, B$ be two dg $k$-algebras. Let $F:\D(A) \ra \D(B)$ be a triangle functor, which commutes with infinite direct sums and restricts to a functor $F|_{\per}: \per (A) \ra \per (B)$. Then $F$ is a triangle equivalence if and only if $F|_{\per}$ is a triangle equivalence. Moreover, in this case, $F$ also restricts to an equivalence $F: \D^{\rm b}(A)\xra{\simeq} \D^{\rm b}(B)$.
\end{Lem}

\begin{Lem}\cite[Lemma 2.7]{AKLY}\label{Lem:pertoDb}
Let $A, B$ be two dg $k$-algebras. Let $F:\D(A) \ra \D(B)$ be a triangle functor with a right adjoint $G:\D(B)\ra \D(A)$. If $F(\per (A))\subset\per(B)$, then $G(\D^{\bb}(B))\subset \D^{\bb}(A)$.
\end{Lem}
\begin{proof}
Let $Y\in \D^{\bb}(B)$. By Lemma \ref{Lem:totalcoh}, it suffices to show that, for any $P\in \per(A)$, $\RshHom_{A}(P, G(Y))$ has finite dimensional total cohomology. Since $(F, G)$ is an adjoint pair, we have
\[ \h^{n}\RshHom_{A}(P, G(Y))=\Hom_{\D(A)}(P, G(Y)[n])=\Hom_{\D(B)}(F(P), Y[n])=\h^{n}\RshHom_{B}(F(P), Y).  \]
Since $F(P)\in \per(B)$ and $Y\in \D^{\bb}(B)$,  we have $\RshHom_{B}(F(P), Y)$ has finite dimensional total cohomology by Lemma \ref{Lem:totalcoh}, and so does $\RshHom_{P}(P, G(Y))$.
\end{proof}

\subsection{Recollements}\label{Recollements}

In this section we recall some results on recollement.
Recollements, introduced by Beilinson, Bernstein and Deligne in the 1980s, provide a powerful tool for studying problems in triangulated categories and algebraic geometry.
Roughly speaking, a recollement consists of three triangulated categories linked by two triangulated functors both of which have left and right adjoint functors. The notion of recollement is an analogue of an exact sequence for triangulated categories, which generalizes derived equivalences and is closely related to homological ring epimorphisms.

\begin{Def}[\cite{BBD}]~\label{recollement}
Let $\mathcal{T^{'}}$,~~$\mathcal{T}$ and $\mathcal{T^{''}}$ be triangulated categories. A \emph{recollement} of $\mathcal{T}$ in terms of $\mathcal{T^{'}}$ and $\mathcal{T^{''}}$ is the following diagram$:$
  \begin{equation}\label{R}
  \xymatrixcolsep{4pc}\xymatrix{\mathcal{T^{'}} \ar[r]^{i_*} &\mathcal{T} \ar@/_1pc/[l]_{i^*} \ar@/^1pc/[l]^{i^!} \ar[r]^{j^*}  &\mathcal{T^{''}} \ar@/^-1pc/[l]_{j_!} \ar@/^1pc/[l]^{j_{*}}
  }
  \end{equation}
satisfying

  \begin{enumerate}[\rm(1)]
    \item $(i^*,~i_*,~i^!)$ and $(j_!,~j^*,~j_*)$ are adjoint triples$;$
    \item $i_*,~j_*,~j_!$ are fully faithful$;$
    \item $j^*\circ i_*=0;$
    \item for any object $X$ in $\mathcal{T}$, there are two  triangles$:$
    \[ \xymatrix{i_*i^!X \ar[r] &X \ar[r] & j_*j^*X \ar[r]& i_{*}i^{!}X[1] } \]
    \[ \xymatrix{j_!j^*X \ar[r] &X \ar[r] & i_*i^*X \ar[r]&j_!j^*X[1]} \]
    where the morphisms are induced by the units and conuits of the adjunctions.
  \end{enumerate}
 \end{Def}

 We give a well-known example of recollement, which follows from \cite{CPS}.
  \begin{Ex}\label{Ex:idem}
  Let $A$ be a ring and let $e\in A$ be an idempotent. Assume the ideal $AeA$ of $A$ is stratifying (\emph{i.e.} $Ae\otimes^\mathbf{L}_{eAe}{eA}\rightarrow AeA$ is an isomorphism). Then the associated ring epimorphism $A \ra AeA$ induces a recollement of derived module categories of the  form
    \[ \xymatrixcolsep{4pc}\xymatrix{\D(A/AeA) \ar[r]^{i_{*}=\rm{inc.}} &\D(A) \ar@/_1.5pc/[l]_{i^{*}=?\ot_{A}^{\bf L}A/AeA} \ar@/^1.5pc/[l]^{i^{!}=\RshHom_{A}(A/AeA, ?)} \ar[r]^{j^{*}=?\ot_{A}^{\bf L}Ae}  &\D(eAe) \ar@/^-1.5pc/[l]_{j_{!}=?\ot_{eAe}^{\bf L}eA} \ar@/^1.5pc/[l]^{j_{*}=\RshHom_{eA e}(Ae, ?)}
  }. \]

 \end{Ex}
Most examples of recollements in the literature are known to be of this form, up to derived equivalence.  We
give a general version of Example 1.6 in the general dg setting in the next section.

Let $\X, \Y$ and $\ZZ$ be three full subcategories of a triangulated category $\T$. We call the triple $(\X, \Y, \ZZ)$ a \emph{TTF triple} of $\T$, if we have two $t$-structures $\T=\X\perp\Y$ and $\T=\Y\perp\ZZ$. It immediately follows that if  $(\X, \Y,\ZZ)$ is a TTF triple, then $\X, \Y$ and $\ZZ$ are all triangulated subcategories of $\T$.
Note that given a recollement \eqref{R}, we have an induced TTF triple $(j_{!}(\T''), i_{*}(\T')) , j_{*}(\T''))$. We say two recollements are equivalent if the TTF triples associated to them coincide.
The following observation is useful (see for example \cite[Section 1.4.4]{BBD} or \cite[Section 9.2]{N3}).
\begin{Prop}\label{Prop:TTF}
Let $\T$ be a $k$-linear triangulated category. There is a one-to-one correspondence between equivalence classes of recollements of $\T$ and TTF triples of $\T$.
\end{Prop}

We need the following lemma for later use.
\begin{Lem}\label{Lem:adjrecol}
Let \eqref{R} be a recollement.  Then
\begin{enumerate}[\rm (1)]
\item $i^{*}$ admits a left adjoint $F$ if and only if $j_{!}$ admits a left adjoint. Moreover, in this case, $F$ is fully faithful, and $(F(\T'), j_{!}(\T''), i_{*}(\T'))$ is a TTF triple.
\item $i^{!}$ admits a right adjoint $G$ if and only if $j_{*}$ admits a right adjoint. Moreover, in this case, $G$ is fully faithful and $(i_{*}(\T'), j_{*}(\T''), G(\T'))$ is a TTF triple.
\end{enumerate}
\end{Lem}
\begin{proof}
See for example \cite[Proposition 9.1.18]{N3} and \cite[Theorem 2.5]{M}, and their duals.
\end{proof}

\subsection{Homological epimorphisms of dg algebras}
Homological epimorphisms were introduced by Geigle and Lenzing in \cite{GL} to characterize when a homomorphism of rings $\phi: R\to S$ induces a full embedding of bounded derived categories $\D^{\rm b}(\mod S) \ra \D^{\rm b}(\mod R)$. The dg generalization of this notion was studied by \cite{P} and \cite{NS}.
In this section, we recall some basic facts on homological epimorphisms of dg algebras.

Let $A$ and $B$ be two dg $k$-algebras and let $f:A\to B$ be a homomorphism of dg algebras.
We may regard $B$ as a dg $A$-bimodule and any dg $B$-module $M$ has a natural dg $A$-module structure induced by $f$.
We call $f$ a \emph{homological epimorphism} if the induced functor $f_{*}:\D(B)\ra \D(A)$ is fully faithful. Let $X\ra A\xra{f} B \ra X[1] $ be the triangle in $\D(A\ot_{k}A^{\rm op})$ induced by $f$. The following result characterizes when a dg homomorphism is a homological epimorphism.

\begin{Lem}\cite[Lemma 4]{NS}  \label{Lem:homepi}
Let $f:A\ra B$ be a homomorphism of dg algebra. The following are equivalent
\begin{enumerate}[\rm (1)]
\item $f$ is a homological epimorphism;
\item $B\ot_{A}^{\bf L}B \cong B$ in $\D(B)$;
\item There is a recollement
    \[ \xymatrixcolsep{4pc}\xymatrix{\D(B) \ar[r]^{f_{*}} &\D(A) \ar@/_1.5pc/[l]_{i^{*}=?\ot_{A}^{\bf L}B} \ar@/^1.5pc/[l]^{i^{!}=\RshHom_{A}(B, ?)} \ar[r]^{j^{*}=?\ot_{A}^{\bf L}X}  &\Loc(X) \ar@/^-1.5pc/[l]_{j_{!}={\rm inc.}} \ar@/^1.5pc/[l]  },\]
  where $\Loc(X)$ is the smallest triangulated subcategory of $\D(A)$ containing $X$ and closed under infinite direct sums.
\end{enumerate}
\end{Lem}

The following example generalizes Example \ref{Ex:idem}.
\begin{Ex}\cite[Proposition 2.10]{KY1}
Let $A$ be a $k$-algebra  and $e\in A$ be an idempotent. Then there is a dg $k$-algebra $B$ and a morphism of dg algebras $f:A\ra B$, such that there is a  recollement
\[ \xymatrixcolsep{4pc}\xymatrix{
\D(B) \ar[r]^{f_{*}}
&\D(A) \ar@/_1.5pc/[l]_{?\ot_{A}^{\bf L}B} \ar@/^1.5pc/[l]^{\RshHom_{A}(B, ?)}
 \ar[r]^{?\ot_{A}^{\bf L}Ae}
&\D(eAe) \ar@/^-1.5pc/[l]_{?\ot_{eAe}^{\bf L}eA} \ar@/^1.5pc/[l]^{\RshHom_{eA e}(Ae, ?)}
  }. \]
  Moreover, $B$ is also non-positive with $\h^{0}(B)=A/AeA$.
 \end{Ex}


\subsection{Compactly generated triangulated categories}\label{Compactly generated triangulated categories}

Let $\T$ be a triangulated category with infinite direct sums. An object $X\in \T$ is called \emph{compact}, if the functor $\Hom_{\T}(X, ?)$ commutes with infinite direct sums.  We denote by $\T^{\rm c}$ the subcategory of $\T$ consisting of compact objects. It is easy to see $\T^{\rm c}$ is a thick subcategory,  that is, it is a  triangulated subcategory of $\T$ closed under direct summands.

We say that $\T$ is \emph{compactly generated} if
$\T$ admits a set of small generators $\cal C$, that is, if
there is a subset $\C$ of compact objects, such that $\T$ coincides with the localising  subcategory $\Loc(\C)$ (which is the smallest sub triangulated category of $\T$ containing $\C$ and closed under infinite sums).

 In a compactly generated triangulated category, all compact objects can be obtained from a set of compact generators, that is, we have the following.

\begin{Thm}\cite[Section 2.2]{Neeman92}\cite[Theorem 5.3]{Keller94}
Let $\T$ be a compactly generated triangulated category and let $\cal C$ be a set of compact generators. Then $\T^{\rm c}=\thick(\cal C)$.
\end{Thm}

\begin{Ex}
Let $A$ be a finite-dimensional algebra over a field $k$. Then the unbounded derived category $\D(A)$ is compactly generated by the free module $A$. So $\D(A)^{\rm c}=\per (A)$ coincides with $\K^{\rm b}(\proj A)$.
\end{Ex}

We need the following lemma later.
\begin{Lem}\cite[Theorems 4.1 and 5.1]{Neeman96} \label{Lem:N2}
Let $F: \cal C\ra \D$ be a triangle functor between two compactly generated triangulated categories. Assume $F$ admits a right adjoint $G: \D\ra  \C$.
Then the following conditions are equivalent.
\begin{enumerate}[\rm(1)]
\item $G$ admits a right adjoint.
\item $F$ sends compact objects to compact objects.
\item $G$ preserves coproducts.
\end{enumerate}
\end{Lem}

The following result is important for our purpose.

\begin{Thm}\cite[Theorem 2.1]{Neeman92} \label{Thm:Ne}
 Let $\T$ be a compactly generated triangulated category and let $\cal R$ be a set of compact objects. Let $\T'':=\Loc(\cal R)$. Then there is a recollement
\[  \xymatrixcolsep{4pc}\xymatrix{\mathcal{T^{'}} \ar[r]^{i_*} &\mathcal{T} \ar@/_1pc/[l]_{i^*} \ar@/^1pc/[l]^{i^!} \ar[r]^{j^*}  &\mathcal{T^{''}} \ar@/^-1pc/[l]_{j_!} \ar@/^1pc/[l]^{j_{*}}
  } \]
  where $\T'$ is also compactly generated. Moreover, $i^{*}$ and $j_{!}$ induce a short exact sequence
  \[ 0 \la \T'^{\rm c} \xleftarrow{\overline{i^{*}}} \T^{\rm c}\xleftarrow{\overline{j_{!}}} \T''^{\rm c} \la 0,\]
that is, $\overline{j_{!}}$ is fully faithful, $\overline{i_{*}}\circ\overline{j_{!}}=0$ and $\overline{i_{*}}$ induces a triangle equivalence $\T^{\rm c}/\T''^{\rm c}\ra \T'^{\rm c}$ up to direct summands.
\end{Thm}

We give the following useful observation.
\begin{Prop}\label{Prop:D^{b}}
Let $A, B, C$ be three  homologically smooth  and proper dg $k$-algebras. Assume there is a recollement
\[  \xymatrixcolsep{4pc}\xymatrix{\D(B) \ar[r]^{i_*} &\D(A) \ar@/_1pc/[l]_{i^*} \ar@/^1pc/[l]^{i^!} \ar[r]^{j^*}  &\D(C). \ar@/^-1pc/[l]_{j_!} \ar@/^1pc/[l]^{j_{*}}
  } \]
Then there is  an induced recollement
\[  \xymatrixcolsep{4pc}\xymatrix{\D^{\bb}(B) \ar[r]^{i_*} &\D^{\bb}(A) \ar@/_1pc/[l]_{i^*} \ar@/^1pc/[l]^{i^!} \ar[r]^{j^*}  &\D^{\bb}(C).\ar@/^-1pc/[l]_{j_!} \ar@/^1pc/[l]^{j_{*}}
  } \]
\end{Prop}

\begin{proof}
Note that by Theorem \ref{Thm:Ne}, the top functors $i^{*}$ and $j_{!}$ restrict to the perfect categories, hence to the bounded derived categories (since they coincide by our assumption and Lemma \ref{Lem:Db=per}).
Also note that by Lemma \ref{Lem:pertoDb}, the middle functors $i_{*}$ and $j^{*}$ restrict to $\D^{\bb}$,
hence to $\per$. Then again by Lemma \ref{Lem:pertoDb}, we know that the bottom functors $i^{!}$ and $j_{*}$ also restrict to $\D^{\bb}$. So the assertion is true.
\end{proof}

\subsection{Graded marked surfaces and graded gentle algebras}\label{subsection:marked surfaces}
In this subsection we briefly recall some results on marked surfaces. We refer to \cite{HKK17,OPS18,APS19,O19, LP20} for more details and examples.
\begin{Def}\label{definition: marked surface}
A pair~$(S,M)$ is called a \emph{marked surface} if
  \begin{enumerate}[\rm(1)]
 \item $S$ is an oriented surface with boundary with
  connected components $\partial S=\sqcup_{i=1}^{b}\partial_i S$;
 \item $M = M^{\gpoint} \cup M^{\rpoint} \cup P^{\gpoint}\cup P^{\rpoint}$ is a finite set of \emph{marked points} on $S$.
       The elements of~$M^{\gpoint}$ and~$M^{\rpoint}$ are marked points on $\partial S$, which will be respectively represented by symbols~$\gpoint$ and~$\rpoint$. Each connected component $\partial_i S$ is required to contain at least one marked point of each colour, where in general the points~$\gpoint$ and~$\rpoint$ are alternating on $\partial_i S$. The elements in $P=P^{\gpoint}\cup P^{\rpoint}$ are marked points in the interior of $S$. We refer to these points as \emph{punctures}, and we will also represent them by the symbols $\gpoint$ and $\rpoint$ respectively.
  \end{enumerate}
\end{Def}

Note that a marked surface is allowed to have empty boundary. If this happens, then we need both $P^{\gpoint}$ and $P^{\rpoint}$ to be non-empty.  We adopt the convention that the following two special marked surfaces, the one corresponding to  a disk without punctures which has only one $\gpoint$-point and one $\rpoint$-point on the boundary, and the second corresponding to  a sphere with only one $\gpoint$-puncture and one $\rpoint$-puncture are   `zero' and if we  cut the surface and obtain such a component, then we can disregard it.

\begin{Def}\label{definition:arcs}
Let $(S,M)$ be a marked surface.
An \emph{$\gpoint$-arc} is a non-contractible curve, with endpoints in~$M^{\gpoint}\cup P^{\gpoint}$, while an \emph{$\rpoint$-arc} is a non-contractible curve, with endpoints in $M^{\rpoint}\cup P^{\rpoint}$.
\end{Def}

On the surface, all curves are considered up to homotopy, and all  intersections of curves are required to be transversal.

\begin{Def}\label{definition: dissection}
Let $(S,M)$ be a marked surface.
\begin{enumerate}[\rm(1)]
 \item An \emph{admissible $\gpoint$-dissection} $\zD$ of $ (S,M)$ is a collection of pairwise non-intersecting (in the interior of $S$) and pairwise distinct $\gpoint$-arcs  on the surface such that the arcs in $\zD$ cut the surface into polygons each of which contains exactly one $\rpoint$-point.
 \item Let $q$ be a common endpoint of two arcs $\ell_i$ and $\ell_j$ in an admissible $\gpoint$-dissection $\zD$, an \emph{oriented intersection} from $\ell_i$ to $\ell_j$ at $q$ is an anticlockwise angle locally from $\ell_i$ to $\ell_j$ based at $q$ such that the angle is in the interior of the surface. We call an oriented intersection a \emph{minimal oriented intersection} if it is not a composition of two oriented intersections of arcs from $\zD$.
 \item A \emph{graded admissible $\gpoint$-dissection} is a pair $(\zD,G)$, where $\zD$ is an admissible $\gpoint$-dissection and $G$ is a set of gradings, that is, for any minimal oriented intersection $\za$, a \emph{grading} is the data of an integer $g(\za)\in \mathbb{Z}$.
 \item We call $(S,M,\zD, G)$ a \emph{graded marked surface}.
  \end{enumerate}
We similarly define \emph{admissible $\rpoint$-dissections} and \emph{graded admissible $\rpoint$-dissections}. \end{Def}

\begin{Rem}
If there is no $\gpoint$-puncture in $(S,M)$, then for any graded admissible dissection $(\zD,G)$, there exists a unique foliation (up to homotopy) on $S$ (with a smooth structure) which induces the grading $G$, see \cite{LP20}.
\end{Rem}

We introduce the following proposition-definition, whose proof is straightforward.

\begin{Prop-Def}\label{definition: dual graded dissection}
Let $(S,M,\zD,G)$ be a graded marked surface.
\begin{enumerate}[\rm(1)]
 \item
There is a unique (up to homotopy) admissible $\rpoint$-dissection $\zD^*$ on $(S,M)$ such that
 each arc $\ell^*$ in $\zD^*$ intersects exactly one arc $\ell$ of $\zD$, which we will call the \emph{dual $\rpoint$-dissection}
 of $\zD$.
  \item
  For any minimal oriented intersection $\za$ of $\zD$ from $\ell_i$ to $\ell_j$, there is a unique minimal oriented intersection $\za^{op}$ of $\zD^*$ from $\ell_j^*$ to $\ell_i^*$, see Figure \ref{figure:graded dual algebra}. Conversely, any minimal oriented intersection of $\zD^*$ arises from this way. We call $\za^{op}$ the \emph{dual minimal oriented intersection} of $\za$.

\item
The  \emph{ dual graded  $\rpoint$-dissection} $(\zD^*,G^*)$   of $(\zD,G)$ is the dual $\rpoint$-dissection $\zD^*$ of $\zD$ with dual of grading $G^*$ defined by setting $g(\za^{op}) = 1-g(\za)$, for any minimal oriented intersection $\za$ and its dual $\za^{op}$.

\end{enumerate}
\noindent We similarly define the dual graded $\gpoint$-dissection of an admissible graded $\rpoint$-dissection.
\end{Prop-Def}

Note that the dual graded dissection is unique (up to homotopy) and it is an involution, that is $(\zD^{**},G^{**})=(\zD,G)$.
Moreover, the definition of the dual grading is motivated by the algebraic dual of a graded gentle algebra considered in subsection~\ref{Section:Koszul}.

\begin{Ex}\label{ex:dissection}
See Figure \ref{figure:ex-dissection} for an example of an admissible $\gpoint$-dissection of a marked surface and its dual dissection and Figure~\ref{figure:graded dual algebra}  for the dual of minimal oriented intersections.
\begin{figure}
\begin{center}
{\begin{tikzpicture}[scale=0.3]
\tikzset{vertex/.style = {style=circle,draw,fill,red,minimum size = 2pt,inner        sep=1pt}}

\draw[thick] (0,0) circle [radius=6];
\draw[thick] (1,-.5) circle [radius=.6];

\draw[dark-green,thick](5.46,-2.46)--(-5,-3.3);
\draw[dark-green,thick](-5,3.3)--(-5,-3.3);
\draw[dark-green,thick](-5,-3.3)--(1,3);
\draw[dark-green,thick](1,3)--(5.46,-2.46);
\draw[dark-green,thick](1,3)--(-5,3.3);
\draw[dark-green,thick](1.4,-.9)--(5.46,-2.46);

\draw (-6,0) node {$\rpoint$};
\draw (-3,1) node {$\rpoint$};
\draw (0,-6) node {$\rpoint$};
\draw (.5,-.1) node {$\rpoint$};
\draw (4.25,4.25) node {$\rpoint$};

\draw[red,thick](-6,0)--(-3,1);
\draw[red,thick](.7,-.1)--(-3,1);
\draw[red,thick](.5,0)--(4.25,4.25);
\draw[bend right,red,thick](4.25,4.25)to(-3,1);
\draw[bend right,red,thick](.4,-.1)to(0,-6);
    \draw[red,thick] plot [smooth,tension=1] coordinates {(.5,0) (.5,-1.5) (3,-1.5) (2,0.5) (.5,0)};

\draw[dark-green] (0,-3.5) node {\tiny$\ell_1$};
\draw[dark-green] (-1,-.2) node {\tiny$\ell_2$};
\draw[dark-green] (-4.5,0) node {\tiny$\ell_3$};
\draw[dark-green] (-2,3.6) node {\tiny$\ell_4$};
\draw[dark-green] (3.3,1) node {\tiny$\ell_5$};
\draw[dark-green] (2.5,-.8) node {\tiny$\ell_6$};

\draw (1,3) node {$\gpoint$};
\draw (5.46,-2.46) node {$\gpoint$};
\draw (1.4,-.9) node {$\gpoint$};
\draw (-5,-3.3) node {$\gpoint$};
\draw (-5,3.3) node {$\gpoint$};
\end{tikzpicture}}
\end{center}
\begin{center}
\caption{An example of an admissible $\gpoint$-dissection $\zD$ of a marked surface and its dual $\zD^*$, where the arcs in $\zD$ are colored by green and the arcs in $\zD^*$ are colored by red.}\label{figure:ex-dissection}
\end{center}
\end{figure}
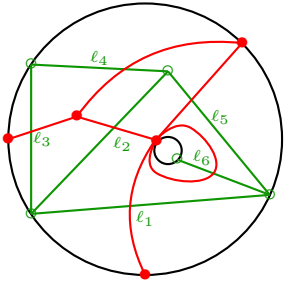

\begin{figure}
 \[\scalebox{1}{
\begin{tikzpicture}[>=stealth,scale=0.6]

\draw[red,thick,fill=red] (2,4) circle (0.1);

\draw[thick,fill=white] (2,0) circle (0.1);
\draw [thick,dark-green] (-1.3,3.8)--(2,0);
\draw [thick,dark-green](5.3,3.8)--(2,0);
\draw [thick,red] (-1.3,.2)--(2,4);
\draw [thick,red](5.3,.2)--(2,4);

\draw [thick,bend right,->] (2.3,.35)to(1.7,.35);
\node at (2,.8) {\tiny$\za$};
\draw [thick,bend left,<-] (2.3,3.65)to(1.7,3.65);
\node at (2,3.2) {\tiny$\za^{op}$};

\node at (1,1.7) {\tiny$\ell_j$};
\node at (3,1.7) {\tiny$\ell_i$};
\node at (1,2.3) {\tiny$\ell^*_j$};
\node at (3,2.3) {\tiny$\ell^*_i$};
\node at (2,-.4) {\tiny$q$};
\node at (2,4.4) {\tiny$q^*$};
\end{tikzpicture}
}\]
\begin{center}
\caption{Dual of minimal oriented intersections $\za$ and $\za^{op}$. We set $g(\za^{op}) = 1-g(\za)$ for the gradings.}\label{figure:graded dual algebra}
\end{center}
\end{figure}
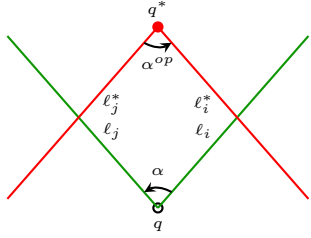
\end{Ex}

For ease of notation, we will often omit the grading $G$ of $(\zD, G)$ and simply write $\zD$ for a graded surface dissection.

We now recall how to construct a graded gentle algebra from an admissible graded dissection of a marked surface.
We begin by recalling the definition of a graded gentle algebra.

\begin{Def}\label{definition:gentle algebras}
A \emph{graded gentle algebra} $A$ is an algebra of the form  $A = K Q/I$  where $Q=(Q_0,Q_1)$ is a finite quiver and $I$ an ideal of $KQ$ satisfying the following conditions:
\begin{enumerate}[\rm(1)]
 \item Each vertex in $Q_0$ is the source of at most two arrows and the target of at most two arrows.

 \item For each arrow $\za$ in $Q_1$, there is at most one arrow $\zb$ such that  $0 \neq \za\zb\in I$; at most one arrow $\zg$ such that  $0 \neq \zg\za\in I$; at most one arrow $\zb'$ such that $\za\zb'\notin I$; at most one arrow $\zg'$ such that $\zg'\za\notin I$.

 \item $I$ is generated by paths of length two.

 \item The grading of $A$ is induced by associating to  each arrow $\alpha$  an integer $| \alpha |$.

\end{enumerate}
\end{Def}
Note that $A$ might be infinite dimensional. Sometimes such an algebra is also referred to as a  \emph{graded locally gentle algebra}.

We now recall how to associate a graded gentle algebra to a graded admissible dissection. For that,
let $(\zD, G)$ be a graded admissible $\gpoint$-dissection on a marked surface $(S,M)$. Define the graded algebra $A(\zD)=kQ(\zD)/I(\zD)$ by setting:

\begin{enumerate}[\rm(1)]
  \item The vertices of $Q(\zD)$ are given by the arcs in $\zD$.
  \item Each minimal oriented intersection $\za$ from $\ell_i$ to $\ell_j$ gives rise to an arrow from $\ell_i$ to $\ell_j$, which is still denoted by $\za$. Furthermore, $|\alpha| = g(\za)$.
  \item  The ideal $I(\zD)$ is generated by the following relations: whenever $\ell_i$ and $\ell_j$ intersect at a marked point which gives rise to an arrow $\za:\ell_i\rightarrow\ell_j$, and the other end of $\ell_j$ intersects $\ell_k$ at a marked point which gives rise to an arrow $\zb:\ell_j\rightarrow\ell_k$, then the composition $\za\zb$ is a relation.
\end{enumerate}

Then $A(\zD)$ is a graded gentle algebra. In particular this establishes a bijection between the set of homeomorphism classes of marked surfaces with graded admissible $\gpoint$-dissections and the set of isomorphism classes of graded gentle algebras \cite{OPS18,PPP18}. Similarly, we construct a graded gentle algebra $A(\zD^*)=kQ(\zD^*)/I(\zD^*)$ from a graded admissible $\rpoint$-dissection $\zD^*$.

\begin{Def}\label{Def:exceptionaldissection}
An $\gpoint$-arc is called an \emph{exceptional $\gpoint$-arc} if $\zg$ has no self-intersections, that is, $\zg$ does not have any self-intersections neither in the interior of $S$ nor at its endpoints.
A $\gpoint$-dissection $\zD$ is called \emph{an $\gpoint$-exceptional dissection} if there is no oriented cycles in the quiver $Q(\zD)$ associated to $\zD$. In particular, there is no loop in $Q(\zD)$ and each arc in $\zD$ is an exceptional $\gpoint$-arc.
\end{Def}

\begin{Ex}\label{ex:algebra}
See Figure \ref{figure:ex-algebra} for the gentle algebra $A(\zD)$ associated  to  the admissible $\gpoint$-dissection $\zD$ depicted in Figure \ref{figure:ex-dissection}. Note that here $\zD$ is not an exceptional dissection since $Q(\zD)$ has an oriented cycle.
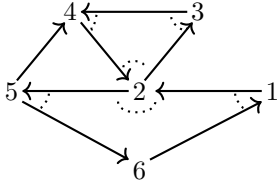
\begin{figure}
\begin{center}
{\begin{tikzpicture}[scale=0.7]
\def \radius {4cm}
    \draw [thick,<-] (-1,0) -- (1,0);
    \draw (-1.2,0) node {4};
    \draw (1.2,0) node {3};

    \draw [thick,<-] (-1-1.1,-1.5) -- (1-1.1,-1.5);
    \draw (-1.2-1.1,-1.5) node {5};
    \draw (1.2-1.1,-1.5) node {2};
    \draw (.1,-3) node {6};
    \draw [thick,<-] (-1+1.4,-1.5) -- (1+1.4,-1.5);
    \draw (1.2+1.4,-1.5) node {1};
    \draw [thick,<-] (1.1,-.2) -- (.2,-1.3);
    \draw [thick,->] (-1,-.2) -- (-.1,-1.3);
    \draw [thick,<-] (1.1-2.4,-.2) -- (.2-2.4,-1.3);

    \draw [thick,<-] (1.2+1.4-.2,-1.5-.2) -- (.1+.2,-3+.2);
    \draw [thick,<-] (.1-.2,-3+.2) -- (-1.2-1.1+.2,-1.5-.2);

    \draw[dotted,thick](-.2,-1.05) to [out=90,in=90] (.2,-1.05);
    \draw[dotted,thick](.7,-.1) to [out=-90,in=-90](1,-.3);
    \draw[dotted,thick](-.8,-.3) to [out=-90,in=-90](-.7,0);

    \draw[dotted,thick](-.3,-1.7) to [out=-90,in=-90] (.4,-1.7);
    \draw[dotted,thick](1.9,-1.55) to [out=-90,in=-90](2.1,-1.8);
    \draw[dotted,thick](-1.6,-1.55) to [out=-90,in=-90](-1.8,-1.8);

\end{tikzpicture}}
\end{center}
\begin{center}
\caption{The gentle algebra associated to the admissible dissection in Figure \ref{figure:ex-dissection}, where the relations are depicted by dotted lines.}\label{figure:ex-algebra}
\end{center}
\end{figure}
\end{Ex}

\section{Recollements arising from graded quadratic monomial algebras}



\subsection{Graded quadratic monomial algebras} \label{Section:monomial}
Let $Q$ be a graded quiver, that is, to every arrow in $Q$ there is  an associated  integer $|\za|$. Furthermore we assume that  $Q_{0}$ a finite set.
Let $I$ be a set of quadratic monomial relations and  let $A=kQ/\langle I \rangle$. Then $A$ is a graded $k$-algebra and we may regard $A$ as a dg $k$-algebra with zero differential.
Let $J$ be a subset of $I$. In this subsection, we construct a new graded quadratic monomial algebra $A_{J}=KQ'/\langle I'\rangle$ such that $A_{J}$ is quasi-isomorphic to $A$ by `killing' the relations in $J$. Note that if $J=I$, then $I'=0$ and  $A_J$ is the dg quiver algebra associated to $A$ sometimes referred to as the cofibrant dg algebra resolution of $A$.

 Define the  following  sets of paths in $Q$.
\[J_{1}:=Q_{1}, \   J_{2}:=J,  \    J_{3}:=\{\alpha_{1}\alpha_{2}\za_{3} \mid \alpha_{1}\za_{2}\in  J, \za_{2}\za_{3}\in J\},\]
 and, in general, we define
\[ J_{n}:=\{\za_{1}\cdots\za_{n}\mid \za_{i}\za_{i+1}\in J, \forall i=1,2,\dots,n-1\}.\]
Note that $J_{n}$ can be an empty set.
\begin{Def}\label{Def:AJ}
We define the \emph{partial cofibrant dg algebra resolution}
 $A_{J}=kQ'/\langle I'\rangle$ as follows.
\begin{enumerate}[\rm(1)]
\item $Q'_{0}=Q_{0}$.
\item The arrows of $Q'$ are given by
\[Q'_{1}=\{[\alpha_{1}\dots\za_{n}]: s(\za_{1})\ra t(\za_{n}) \mid \za_{1}\dots\za_{n}\in J_{n},  n\ge 1 \},\]
where  $|[\za_{1}\dots\za_{n}]|:=\sum_{i=1}^{n} |\alpha_{i}|-n +1$.
\item The quadratic relations are given by
\[I':=   \{[\za_{1}\dots\za_{n}][\zb_{1}\dots\zb_{m}] \mid \za_{n}\zb_{1}\in I\backslash J, n,m\ge 1 \}.   \]
\item The differential $d'$ is defined as $d'([\alpha])=0$ for any $\alpha\in Q_{1}$ and
\[d'([\za_{1}\dots\za_{n}]):=\sum_{i=1}^{n-1}(-1)^{|[\za_{1}\dots\za_{i}]|}[\za_{1}\dots\za_{i}][\za_{i+1}\dots\za_{n}],\]
for any  $\za_{1}\dots\za_{n} \in J_{n}$ and $n\ge 2$.
\end{enumerate}
\end{Def}
One can check that $d'^{2}=0$, so $A_{J}$ is a well-defined dg $k$-algebra.
\begin{Rem}\label{Rem:dgquiver}
\begin{enumerate}[\rm (1)]
\item If $J=I$, then $I'=\emptyset$ and in this case, $A_{J}=(kQ', d')$ is a dg quiver algebra, which can be regarded as a cofibrant dg algebra resolution of $A$.
Moreover, any algebra admits a dg algebra resolution (see for example \cite{O}) and for a monomial algebra, there is an explicit minimal choice given in  \cite{T},  which coincides with our construction above for the case $J=I$.
\item
If $J=\emptyset$, then $J_{n}=\emptyset$ for any $n\ge 2$ and moreover, $A_{J}=A$.
\end{enumerate}
\end{Rem}

We illustrate the construction above in an example.

\begin{Ex}
Let $Q: 1\xra{\za} 2 \xra{\zb} 3 \xra{\zd} 4 \xra{\ze} 5$ be the graded  quiver with zero grading and where $I=\{\za\zb, \zb\zd, \zd \ze \}$. Let $J=\{\za\zb, \zb\zd \}$. Then $Q'$ is

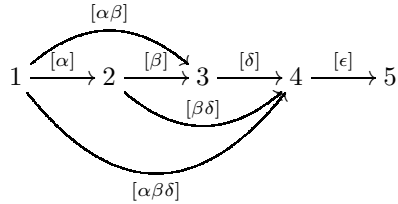
\[\xymatrix{ 1 \ar[r]^{[\za]} \ar@/^1.5pc/[rr]^{[\za\zb]} \ar@/^-3pc/[rrr]_{[\za\zb\zd]}& 2 \ar[r]^{[\zb]} \ar@/^-1.5pc/[rr]^{[\zb\zd]}& 3 \ar[r]^{[\zd]}  & 4 \ar[r]^{[\ze]} & 5 }   \]
where  the grading is given by $|[\za\zb]|=-1=|[\zb\zd]|$, $|[\za\zb\zd]|=-2$ and the differential is given by $d'([\za\zb])=[\za][\zb]$, $d'([\zb\zd])=[\zb][\zd]$ and $d'([\za\zb\zd])=[\za][\zb\zd]-[\za\zb][\zd]$ and where $I' = \{ [\zd] [\ze], [\zb\zd][\ze],  [\za\zb\zd][\ze] \}$.
\end{Ex}

Note that in some cases $A_J$ might be a quiver with infinitely many arrows. We give an example of such a case.

\begin{Ex}
Let $Q$ be the graded quiver with zero grading $\xymatrix{1\ar@(ur,dr)^{\za}}$ and let $I=\{\za^{2}\}$. Let $J=I$, then $Q'$ is

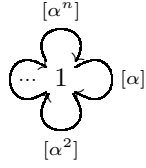
\[\xymatrix{ 1 \ar@(ur,dr)^{[\za]} \ar@(dr,dl)^{[\za^{2}]} \ar@(ul,ur)^{[\za^{n}]} \ar@(dl,ul)_{...}
}\]
where $|[\za^{n}]|=1-n$ and $d'(|\za^{n}|)=\sum_{i=1}^{n-1}(-1)^{1-i}[\za^{i}][\za^{n-i}]$.
\end{Ex}

Our main result in this subsection is the following.

\begin{Thm}\label{Thm:AJA}
There is a natural quasi-isomorphism $\phi:A_{J}\ra A$ of dg $k$-algebras.
\end{Thm}
\begin{proof}
We first define a map $\phi: Q'_{0}\cup Q_{1}'\ra kQ$ by $\phi(e_{j}):=e_{j}$, $\phi([\za]):=\za$ for any $\za \in Q_{1}$,
$\phi([\za_{1}\cdots\za_{n}]):=0$ for $n\ge 2$.
Then $\phi$ can be extended to an epimorphism $\phi: kQ' \ra kQ$ of graded $k$-algebras.
Note that we have $\phi(I')\subset I$ and $\phi( d'([\za_{1}\cdots\za_{n}]))\in I$, for any $n\ge 1$. Thus $\phi$
 also induces an epimorphism $\phi: A_{J}\ra A$ of dg algebras.

 Next we show that $\phi$  is a quasi-isomorphism of $k$-complexes. Let $P:=\ker \phi$. It is enough to show $P$ is acyclic. Define a map $N: Q_{1}\times Q_{1} \ra \{0,1\}$ by
 \[N(\za, \zb):=\begin{cases} 1 & \mbox{ if } \za\zb \in J;\\
 0 & \mbox{ if } \za\zb\not\in  J.\end{cases} \]

 We construct a morphism $\psi: P \ra P[-1]$ of graded $k$-modules  by
 \begin{align*} & \psi([\za_{1}\cdots\za_{t_{1}}][\za_{t_{1}+1}\cdots\za_{t_{2}}]\cdots[\za_{t_{s}+1}\cdots\za_{n}])
 \\
 :=& \frac{1}{m}\sum_{i=1}^{s}(-1)^{\sum_{j=1}^{i}|[\za_{t_{j-1}+1}\cdots\za_{t_{j}}]|} [\za_{1}\cdots\za_{t_{1}}][\za_{t_{1}+1}\cdots\za_{t_{2}}]\cdots[\za_{t_{i-1}+1}\cdots\za_{t_{i}}\za_{t_{i}+1}\cdots\za_{t_{i+1}}]\cdots[\za_{t_{s}+1}\cdots\za_{n}],
 \end{align*}
where $m=n-s-1+\sum_{i=1}^{s}N(\za_{t_{i}},\za_{t_{i}+1})$ and where we set $t_0 =0$.
If $\za_{t_{i}}\za_{t_{i}+1}\not\in J$, we set above term to be zero.

We claim that $d'\psi+\psi d'=\rm{Id}$. Then $P$ is contractible, and hence acyclic.
First note that any element $l_{1}+\cdots+l_{t}$ of $kQ'$ is in $P$ if and only each $l_{i}$ is in $P$. So we only need to check $d\psi+\psi d'=\rm{Id}$ for paths in $P$.

Given $l=[\za_{1}\cdots\za_{t_{1}}][\za_{t_{1}+1}\cdots\za_{t_{2}}]\cdots[\za_{t_{s}+1}\cdots\za_{n}]\in P$, $l$ appears $(n-s-1)/m$ times in $\psi d'(l)$ and appears $(\sum_{i=1}^{s}N(\za_{t_{i}},\za_{t_{i}+1}))/m$ times in $d'\psi(l)$. For any other term $l'$, if it appears in $\psi d'(l)$, one can check directly that the term $-l'$ appears in $d'\psi(l)$.  So $d'\psi(l)+\psi d'(l)=l$.
\end{proof}

\begin{Rem}
Note that in the proof of Theorem~\ref{Thm:AJA}
we need the assumption that the characteristic of $k$ is zero.
\end{Rem}

 Next we study the graded quadratic dual of $A$.
 Let $I^{\perp}:=\{ \zb^{\rm op}\za^{\rm op}\mid \za\zb \notin I\}$. Let $A^{!}=kQ^{\rm op}/I^{\perp}$ such that
$|\za^{\rm op}|:=1-|\za|$ for any arrow $\za^{\rm op}\in Q^{\rm op}_{1}$.
It is obvious that $A^{!}$ is also a graded quadratic monomial algebra and  $(A^{!})^{!}=A$. Let $S_{i}$ be the simple dg $A$-module corresponding to the vertex $i\in Q_{0}$ and let   $S:=\bop_{i\in Q_{0}}S_{i}$.
The following observation is known and we add a proof for the convenience of the  reader.

\begin{Prop}\label{Prop:Koszuldual}
$A^{!}$ is isomorphic to  the graded algebra $\bop_{t\in \Z}\Hom_{\D^{\rm b}(A)}(S, S[t])$.
\end{Prop}

\begin{proof}
To calculate $\Hom_{\D^{\rm b}(A)}(S, S[t])$, we first give a projective resolution of $S$.
Let $M_{0}:=\bop\limits_{i\in Q_{0}}Ae_{i}\ot e_{i}A$ and
\[M_{m}:=\bop\limits_{\za_{1}\cdots\za_{m}\in\cal F_{m}}Ae_{s(\za_{1})}\ot e_{t(\za_{m})}A(|\za_{1}|+\cdots+|\za_{m}|),\]
where $\cal F_{m}:=\{\za_{1}\cdots \za_{m} \mbox{ paths of $kQ$ such that } \za_{i}\za_{i+1}\in I\}$ and we denote by $(|\za_{1}|+\cdots+|\za_{m}|)$ the degree shift of $Ae_{s(\za_{1})}\ot e_{t(\za_{m})}A$  as graded $A$-$A$-bimodule.
There is an exact sequence of graded $A$-$A$-bimodules
\begin{equation}\label{equ:AAresol}  \cdots \ra M_{2} \xra{d_{2}} M_{1} \xra{d_{1}} M_{0}  \xra{d_{0}} A \ra 0  \end{equation}
which gives a minimal $A$-$A$-bimodule resolution of $A$ (see for example \cite[Theorem 4.1]{B}), where $d_{m}(a\ot b):=a\za_{1}\ot b+ (-1)^{m}a\ot \za_{n}b$ and $d_{0}(a\ot b):=ab$.
Since \eqref{equ:AAresol} is contractible,
applying $S_{i}\ot_{A}?$ to \eqref{equ:AAresol}, we have  a minimal resolution of $S_{i}$ as graded $A$-module
\begin{equation}\label{resolofsimple}
P:    \  \cdots \ra \bop\limits_{\za_{1}\za_{2}\in \cal F_{2} \atop s(\za_{1})=i}(e_{t(\za_{2})}A(|\za_{1}|+|\za_{2}|)) \xra{\partial_{2}} \bop\limits_{\za\in Q_{1} \atop s(\za)=i}(e_{t(\za)}A(|\za|)) \xra{\partial_{1}} e_{i}A \ra 0,
  \end{equation}
  whose total complex $\Tot(P)$ gives a projective resolution of $S_{i}$ as dg $A$-module, where $\partial_{m}(x)=(-1)^{m}x\za_{m}$. Then any non-zero morphism from $S_{i}$ to $S_{j}[l]$ in $\D^{\rm b}(A)$ is induced by some non-zero map from $\Top (P^{m})$ to $S_{j}[l]$. Moreover,  we see that $\Hom_{\D^{\rm b}(A)}(S_{i}, S_{j}[l])$ has a $k$-basis
  \[\{\za_{1}\cdots\za_{p}\in \cal F_{p} \mbox{ such that $s(\za_{1})=i, t(\za_{p})=j$ and } p=l+|\za_{1}|+\cdots+|\za_{p}|\}.\]
  In particular, any arrow $\za\in Q_{1}$ gives a non-zero morphism $f_{\za}: S_{s(\za)} \ra S_{t(\za)}[1-|\za|]$. By the construction of $P$, we have $(\f_{\zb}[1-|\za|])\circ f_{\za}\not=0$ if and only if $0\not=\za\zb\in I$.
  Thus the map $\phi: A^{!}\ra \bop_{t\in \Z}\Hom_{\D^{\rm b}(A)}(S, S[t])$ sending $e_{i}$ to  the identity morphism in  $\End_{\D^{\rm b}(A)}(S_{i})$ and $\za^{\rm op}$ to $f_{\za}$ gives an isomorphism of dg algebras.
   \end{proof}

\subsection{Recollements induced by idempotents of graded quadratic monomial algebras}\label{Section:Koszul}
Let $A=kQ/I$ be a graded quadratic monomial algebra. Let $e$ be the sum of $m$ vertex idempotents. Without loss  of generality, we may write  $e=e_{1}+\dots+e_{m}$.
Let $J_{e}$ be the subset of $I$ consisting of all the relations go through vertices $1, \dots, m$.

\begin{Def}\label{definition:left algebra}
We define a new algebra $A_{e}:=kQ_{e}/\langle I_{e}\rangle$ as follows
\begin{enumerate}[\rm(1)]
\item The set of vertices of $Q_{e}$ is $Q_{0}\backslash \{1,2,\dots, m\}$.
\item The arrows of $Q_{e}$ are of the form $[\za_{1}\cdots\za_{s}]:s(\za_{1})\ra t(\za_{s})$ with $\za_{1}\cdots\za_{s}\in J_{s}$ and $s(\za_{1}), t(\za_{s})\not\in \{1,2,\cdots,m \}$.
\item $|[\za_{1}\cdots\za_{s}]|:=\sum_{i}^{s}|\za_{i}|-s+1$.
\item $I_{e}:=\{[\za_{1}\cdots\za_{s}][\zb_{1}\cdots\zb_{t}]\mid \za_{s}\zb_{1}\in I  \}$.
\item The differential is $0$.
\end{enumerate}
\end{Def}

\begin{Ex}\label{Ex:gentleA_e}
In case $A$ is a graded gentle algebra, the algebra $A_e = kQ_e /I_e$  in Definition~\ref{definition:left algebra} is again a graded gentle algebra. In case $e = e_\ell$ is an idempotent corresponding to a single vertex $\ell$, we obtain the following explicit description.
\begin{enumerate}[\rm (1)]
\item The vertex set of $Q_e$ is given by $Q_0 \setminus \{\ell\}$;
\item The arrow set of $Q_e$ is the union
$$\{ [\za]=\ell_i \xra{\za} \ell_j \mid \ell_i \not=\ell \text{ and } \ell_j\not= \ell \}~\sqcup~\{ [\za_1\za_2]=\ell_i \xra{\za_1} \ell \xra{\za
_2} \ell_j \mid \za_1 \za_2\in I\ \mbox{ and } \ell_i \neq \ell  \mbox{ and } \ell_j \neq \ell\}$$
$$~\sqcup~\{ [\za_1 \za_2 \za_3]=\ell_i \xra{\za_1} \ell\xra{\za_2} \ell \xra{\za
_3} \ell_j \mid \za_1\za_2, \za_2\za_3\in I\};$$
\item $|[\alpha]| = |\alpha|$, $|[\za_1 \za_2]| = |\za_1|+|\za_2|-1$ and $|[\za_1 \za_2\za_3]| = |\za_1|+|\za_2|+|\za_3|-2$;
\item The ideal $I_\ell$ is generated by
\[ \{[\za_1][\za_2]\mid \za_1\za_2\in I\}~\sqcup~\{[\za_0][\za_1\za_2]\mid \za_0\za_1\in I \}~\sqcup~\{[\za_1\za_2][\za_3]\mid \za_2\za_3\in I \}\]
\[~\sqcup~\{[\za_0][\za_1\za_2\za_3]\mid \za_0\za_1\in I \}~\sqcup~\{[\za_1\za_2 \za_3][\za_4]\mid \za_3\za_4\in I \}~\sqcup~\{[\za_0\za_1][\za_2\za_3]\mid \za_1\za_2 \in I\}.\]
\end{enumerate}
\end{Ex}

We show

\begin{Thm}\label{Thm:recollement1}
There is a recollement
\begin{equation} \label{mainrecoll}
\xymatrixcolsep{4pc}\xymatrix{
\D(A_{e}) \ar[r]^{i_{*}=i_{!}}
&\D(A) \ar@/_1.5pc/[l]_{i^{*}} \ar@/^1.5pc/[l]^{i^{!}}
 \ar[r]^{j^{*}=?\ot_{A}^{\bf L}Ae}
&\D(eAe) \ar@/^-1.5pc/[l]_{j_{!}=?\ot^{\bf L}_{eAe}eA} \ar@/^1.5pc/[l]^{j_{*}=\RshHom_{eAe}(Ae, ?)}
  }, \end{equation}
 such that $i^{*}$ sends $e_{j}A$ to $e_{j}A_{e}$ and $i_{*}$ sends $\Top(e_{j}A_{e})$ to $\Top (e_{j}A)$ for any $j\not=1,2,\dots, m$.
\end{Thm}

\begin{Rem}
Note that if $J_{e}=I$,
then \eqref{mainrecoll} is given by \cite[Theorem 7.1]{KY}. One may  also use Drinfeld's construction of dg quotient \cite{Drinfeld} to give a dg algebra  $C$ such that $\D(C)$ fits the left side of \eqref{mainrecoll}. We point out that $C$ is quasi-isomorphic to $A_{e}$, but  in the case that $A$ is gentle, $C$ is not gentle anymore and the structure of $C$ may be complicated.
\end{Rem}

Recall  that we constructed a dg algebra $A_{J_{e}}$ from $J_{e}$ (see Definition \ref{Def:AJ}). To simplify the notation, we write $B=A_{J_{e}}$.
The following observation is important for us.
\begin{Lem} \label{Lem:A'recB}
We have $A_{e}=B/BeB$ and
the natural morphism $B \ra A_{e}$ is a homological epimorphism of dg $k$-algebras. Moreover, there is a recollement
    \[ \xymatrixcolsep{4pc}\xymatrix{\D(A_{e}) \ar[r]^{\w{i_{*}}=\rm{inc.}} &\D(B) \ar@/_1.5pc/[l]_{\w{i^{*}}=?\ot_{B}^{\bf L}A_{e}} \ar@/^1.5pc/[l]^{\w{i^{!}}=\RshHom_{B}(A_{e}, ?)} \ar[r]^{\w{j^{*}}=?\ot_{B_{{e}}}^{\bf L}B_{{e}}e}  &\D(eBe) \ar@/^-1.5pc/[l]_{\w{j_{!}}=?\ot_{eB_{{e}}e}^{\bf L}eB_{{e}}} \ar@/^1.5pc/[l]^{\w{j_{*}}=\RshHom_{eB e}(Be, ?)}
  }. \]
\end{Lem}

\begin{proof}
It suffices to show $A_{e} \ot_{B}^{\bf L} A_{e} \cong A_{e}$ in $\D(A_{e})$ by Lemma \ref{Lem:homepi} (2).
Notice that we have the following canonical triangle in $\D(B\ot_{k}B^{\rm op})$
\[BeB \ra B \xra{p} A_{e} \ra BeB[1].  \]
Applying $?\ot_{B}^{\bf L} A_{e}$ to this triangle, we get a triangle in $\D (A_{e})$
\[ BeB\ot_{B}^{\bf L}{ A_{e}} \ra A_{e} \ra A_{e}\ot_{B}^{\bf L} A_{e}\ra BeB\ot_{B}^{\bf L}A_{e}[1].\]
We claim that $BeB\ot_{B}^{\bf L}A_{e} =0$, then $A_{e} \cong A_{e}\ot_{B}^{\bf L}A_{e}$ in $\D(A_{e})$ holds.
Note that  $BeB=\bop_{s\in Q_{0}}e_{s}BeB$ as right dg $B$-module. Let $L_{s}$ be the set of paths from $s$ to $i$, which are not in $I$. By our construction of $Q'$, there is no relation at $i$, so we have $le_{i}B\cong e_{i}B$ as right dg $B$-module for any $l\in L_{s}$. So $BeB=\bop_{s\in Q_{0}}\bop_{l\in L_{s}} le_{i}B$ is a  projective module.
 Then $BeB\ot_{B}^{\bf L}A_{e}=BeB\ot_{B}B/BeB=0$. Then the claims follows. So the assertion is true.
\end{proof}

Now we show Theorem \ref{Thm:recollement1}.
\begin{proof}[Proof of Theorem \ref{Thm:recollement1}]
By Theorem \ref{Thm:AJA}, there is a quasi-isomorphism $\phi: B \ra A$, it induces a quasi-isomorphism $eBe \ra eAe$.
Then we have triangle equivalences $\D(A) \xra[\simeq]{\phi^{*}} \D(B)$ and $\D(e_{i}Ae_{i})\xra[\simeq]{\phi^{*}} \D(e_{i}Be_{i})$. Consider the recollement constructed in Lemma  \ref{Lem:A'recB}.
Let $i^{*}:=\w{i^{*}}\phi^{*}$, $i_{*}:=(\phi^{*})^{-1}\w{i_{*}}$ and $i^{!}:=\w{i^{!}}\phi^{*}$.
Then we get the recollement \eqref{mainrecoll}. Since $\phi$ induces a quasi-isomorphism $e_{j}B\ra e_{j}A$ and $\w{i_{*}}(\Top(e_{j}A_{e}))=\Top(e_{j}B)$, we have
\[i^{*}(e_{j}A)=\w{i^{*}}\phi^{*}(e_{j}A)\cong \w{i^{*}}(e_{j}B)=e_{j}B\ot_{B}^{\bf L}A_{e}=e_{j}A_{e}\]
and
\[ i_{*}(\Top(e_{j}A_{e}))=(\phi^{*})^{-1}\w{i_{*}}(\Top(e_{j}A_{e}))=(\phi^{*})^{-1}(\Top(e_{j}B))=\Top(e_{j}A),  \]
 for any $j\not=1,2,\dots, m$. So the assertion is true.
\end{proof}

\begin{Ex}
Let $Q$ be the graded quiver $\xymatrix{1 \ar[r]^{\za} & 2 \ar[r]^{\zg} \ar@(ul,ur)^{\zb}& 3}$. Let $I:=\{\za\zb,\zb^{2},\zb\zg\}$ and $A=kQ/\langle I \rangle$. Let $e=e_{2}$. Then $J_{e}=I$ and $A_{e}$ is given by the following `infinite Kronecker quiver'
\[\xymatrixcolsep{5pc}\xymatrix{
1 \ar@<-.5pc>[r]_{[\za\zb\zg]}   \ar@<.5pc>[r]_{[\za\zb^{2}\zg]}  \ar@<1.5pc>[r]_{[\za\zb^{3}\zg]}^{\cdots}&3
 }\]
where $|[\za\zb^{n}\zg]|:=|\za|+n|\zb|+|\zg|-n-1$.
\end{Ex}

 We give a useful observation.

\begin{Prop}\label{Prop:e'e''}
Let $e=e'+e''$. Then we have isomorphisms  $(A_{e'})_{e''}\cong A_{e}\cong (A_{e''})_{e'}$ of dg algebras.
\end{Prop}
\begin{proof}
We only show that there is an isomorphism $(A_{e'})_{e''}\cong A_{e}$ of dg $k$-algebras.
We may assume $e'=e_{1}+\cdots+e_{l}$ and $e''=e_{l+1}+\cdots+e_{m}$ for $1\le l<m$.
First note that the induced quivers $Q_{e}$ and $(Q_{e'})_{e''}$ are isomorphic. In fact, they have same sets of vertices given by $Q_{0}\backslash \{1,2,\dots,m \}$. Moreover, any arrow in $A_{e}$ has the form $[\za_{1}\cdots\za_{m}]$ such that $\za_{i}\za_{i+1}\in J_{e}$ for any $1\le i\le m-1$. Set $p_{0}=0$, since $J_{e}=J_{e'}\sqcup J_{e''}$, we may write
\begin{equation}\label{equ:Aee}
[\za_{1}\cdots\za_{m}]=[(\za_{p_{0}+1}\cdots\za_{p_{1}})(\za_{p_{1}+1}\cdots\za_{p_{2}})\cdots(\za_{p_{l}+1}\cdots\za_{m})],
\end{equation}
where $\za_{p_{j}+q(j)}\za_{p_{j}+q(j)+1}\in J_{e'}$ and $\za_{p_{j+1}}\za_{p_{j+1}+1}\in J_{e''}$ for any $0\le j\le l-1$ and $1\le q(j)\le p_{j+1}-1$. It is clear that \eqref{equ:Aee} gives rise to a bijection between the set of arrows in $Q_{e}$ and that in $(Q_{e'})_{e''}$.  So $Q_{e}$ and $(Q_{e'})_{e''}$ are isomorphic. Also it is easy to check that $I_{e}$ and $(I_{e'})_{e''}$ are also isomorphic by Definition \ref{definition:left algebra}(4). Thus $A_{e}$ is isomorphic to $(A_{e'})_{e''}$.
\end{proof}

We can describe $A_{e}$ by using $A^{!}$.

We have the following useful observations.
\begin{Prop}\label{Prop:Koszul}
There are isomorphisms of dg $k$-algebras.
\begin{enumerate}[\rm (1)]
\item $A_{e}\cong ((1-e)A^{!}(1-e))^{!}$.
 \item $eAe\cong ((A^{!})_{1-e})^{!}$.
 \end{enumerate}
\end{Prop}
\begin{proof}
We only show $(1)$, since $(2)$ follows directly from $(1)$ and $(A^{!})^{!}=A$.

First note that both algebras have  the same vertices.
Assume $Q_{0}={1,2, \dots, n}$. Then $1-e=e_{m+1}+\cdots+e_{n}$. Any arrow $l^{\rm op}$ in  $(1-e)A^{!}(1-e)$ is given by a path $l^{\rm op}=\za_{t}^{\rm op}\cdots\za_{1}^{\rm op}$ with degree $(t-\sum_{i=1}^{t}|\za_{i}|, 1)$ such that $s(\za_{t}^{\rm op}), t(\za_{1}^{\rm op}) \in \{m+1,\dots, n\}$ and $\za_{j}^{\rm op}\za_{j-1}^{\rm op}\not\in J_{e}^{\perp}$ for any $2\le j\le t$. Then $l=\za_{1}\cdots\za_{t}$ gives rise to an arrow of $((1-e)A^{!}(1-e))^{!}$ with $\deg l =(\sum_{i=1}^{t}|\za_{i}|-t+1, 1)$. Moreover, $s(\za_{1}), t(\za_{t})\not\in\{1,\dots,m\}$  and $\za_{j}\za_{j+1}\in J_{e}$ for any $1\le j\le t-1$. Then $l$ also gives rise to an arrow of $A_{e}$.

Let $l'=\zb_{1}\cdots\zb_{s}$ be another arrow.  Then $ll'=0$ in $((1-e)A^{!}(1-e))^{!}$ if and only if $l'^{\rm op}l^{\rm op}\not=0$ in $(1-e)A^{!}(1-e)$, if and only if $\za_{t}\zb_{1}\in I\backslash J_{e}$, if and only if $ll'=0$ in $A_{e}$.

So the two algebras have same arrows (with same grading) and same relations. Thus they are isomorphic.
\end{proof}

\section{Recollements arising from cutting surfaces }\label{Cut the graded marked surfaces}

In this section, we apply Theorem~\ref{Thm:recollement1} to the class of graded gentle algebras. In particular, we show that  both the left and right side of the recollement in Theorem~\ref{Thm:recollement1} correspond to cutting the   surface associated to the graded gentle algebra along arcs of the corresponding graded  surface  dissection and its dual, respectively.

More precisely, let
$A(\zD)$ be a graded gentle algebra associated to the graded marked surface $(S,M, \zD)$ and let $L$ be a subset of the arcs in the surface dissection $\zD$.
 In the following we define the  graded cut marked surface $(S_{L}, M_{L},\zD_{L})$ obtained from $(S,M, \zD)$  by cutting along the arcs in $L$.
We can also consider the graded marked surface obtained by cutting along arcs in $\zD^*$, for the dual graded surface dissection $\zD^*$ of $(S,M,\zD)$.

We begin by stating the main result of this section as well as some consequences before introducing the necessary definitions and proving  the stated results.

\begin{Thm}\label{main theorem cutting surfaces}
Let $(S,M, \zD)$ be a  graded marked surface with  associated graded gentle algebra $A(\zD)$. Let $L$ be a subset of $\zD$.
Then we have the following recollement
\begin{equation}\label{eq:recollement}
\xymatrixcolsep{4pc}\xymatrix{
\D(A(\zD_{L})) \ar[r]^{}
&\D(A(\zD)) \ar@/_1.5pc/[l]_{} \ar@/^1.5pc/[l]^{}
 \ar[r]^{}
&\D(A(L)) \ar@/^-1.5pc/[l]_{} \ar@/^1.5pc/[l]^{},
  } \end{equation}
where $\zD_{L}$ is the induced graded surface dissection obtained by cutting $(S,M, \zD)$ along  $L$ and where $L$ is realized  as the graded surface dissection of  the surface $(S_{\zD^* \setminus L^*},M_{\zD^* \setminus L^*})$ which  is  obtained  by cutting  $(S, M, \zD^*)$ along $\zD^* \setminus L^*$ and dualising the induced surface dissection $L^{*}$.
\end{Thm}

We note that in \cite[Section 1.3]{A} a similar  recollement to the above has been suggested by considering a cofibrant dg algebra resolution of $A(\zD)$ which  is also related to surface cuts.

We now give an example of Theorem \ref{main theorem cutting surfaces}.

\begin{Ex}\label{example:recollement-cut-surface}
Consider the algebra $1 \xra{\alpha} 2 \xra{\beta} 3 \xra{\gamma} 4 \xra{\delta} 5$ with $\alpha \beta, \gamma\delta \in I$ corresponding to the surface dissection $\zD = \{ \ell_1, \ell_2, \ell_3, \ell_4, \ell_5 \}$ in Figure~\ref{figure:recollement-cut-surface} and with $L=\{\ell_2, \ell_4\}$.  Suppose further that $|\alpha| = a, |\beta| = b$,  $|\gamma| = c$ and $|\delta |= d$.  Then the surfaces corresponding the left and the right side of the recollement in Theorem \ref{main theorem cutting surfaces} are given in Figure~\ref{figure:recollement-cut-surface}. We note that the arrows $[\alpha \beta] : \overline{\ell_1} \to \overline{\ell_3}$  and $[\gamma \delta] : \overline{\ell_3} \to \overline{\ell_5}$ in $A(\zD_L)$ have degrees $|[\alpha \beta]| = a+b-1$ and $|[\gamma \delta]| = c+d-1$, respectively.  The induced arrow $ \beta \gamma : \ell_2 \to \ell_4$ in $A(L)$ has degree $b+c$.

\begin{figure}[ht]
\begin{center}
\begin{tikzpicture}[scale=0.45]
    \draw (-7,-.5) node {};
	\draw (0,0) circle (4cm);
    \draw (0,-4) coordinate (V1) node {$\gpoint$};
    \draw (0,4) coordinate (V6) node {$\gpoint$};
    \draw (-4,0) coordinate (V9) node {$\rpoint$};
    \draw (4,0) coordinate (V3) node {$\rpoint$};
    \draw (-2,3.5) coordinate (V7) node {$\rpoint$};
    \draw (2,3.5) coordinate (V5) node {$\rpoint$};
    \draw (-2,-3.5) coordinate (V70) node {$\rpoint$};
    \draw (2,-3.5) coordinate (V50) node {$\rpoint$};
    \draw (3.5,-2) coordinate (V2) node {$\gpoint$};
    \draw (3.5,2) coordinate (V4) node {$\gpoint$};
    \draw (-3.5,-2) coordinate (V10) node {$\gpoint$};
    \draw (-3.5,2) coordinate (V8) node {$\gpoint$};

    \draw[dark-green,thick,bend right](V10)to(V8)to(V6)to(V4)to(V2);
    \draw[dark-green,thick](V1)to(V6);
    \draw[red,thick](V70)to(V9);
    \draw[red,thick](V70)to(V7);
    \draw[red,thick](V50)to(V5);
    \draw[red,thick](V50)to(V3);
    \draw[red,thick,bend right](V50)to(V70);
    \draw (-3.2,0) node {\tiny$\ell_1$};
    \draw (3.3,0) node {\tiny$\ell_5$};
    \draw (-1.2,3.2) node {\tiny$\ell_2$};
    \draw (-.5,1.5) node {\tiny$\ell_3$};
    \draw (-.5,-2.5) node {\tiny$\ell^*_3$};
    \draw (1.2,3.2) node {\tiny$\ell_4$};
    \draw (-2.5,-1.8) node {\tiny$\ell^*_1$};
    \draw (2.5,-1.8) node {\tiny$\ell^*_5$};
    \draw (-1.5,0) node {\tiny$\ell^*_2$};
    \draw (1.5,0) node {\tiny$\ell^*_4$};
    \draw[thick,bend right,->](-3.2,1.3)to(-2.8,2);
    \draw (-2.5,1.5) node {\tiny$\za$};
    \draw[thick,bend right,->](2.8,2)to(3.2,1.3);
    \draw (2.5,1.5) node {\tiny$\delta$};
    \draw[thick,bend right,->](-.6,3.2)to(0,3.1);
    \draw (-.45,2.6) node {\tiny$\zb$};
    \draw[thick,bend right,->](0,3.1)to(.6,3.2);
    \draw (.4,2.6) node {\tiny$\zg$};
    \draw (0,-5.5) node {(A)$~~~~(S,M,\zD)$ corresponding to  $A(\zD)$};
\end{tikzpicture}
\begin{tikzpicture}[scale=0.45]
	\draw (0,0) circle (4cm);
    \draw (0,-4) coordinate (V1) node {$\gpoint$};
    \draw (0,4) coordinate (V6) node {$\gpoint$};
    \draw (-4,0) coordinate (V9) node {$\rpoint$};
    \draw (4,0) coordinate (V3) node {$\rpoint$};
    \draw (3.5,-2) coordinate (V2) node {$\gpoint$};
    \draw (-3.5,-2) coordinate (V10) node {$\gpoint$};
    \draw (-1.8,0) node {\tiny$\overline{\ell_1}$};
    \draw (1.8,0) node {\tiny$\overline{\ell_5}$};
    \draw (-.5,0) node {\tiny$\overline{\ell_3}$};
    \draw[dark-green,thick](V10)to(V6);
    \draw[dark-green,thick](V2)to(V6);
    \draw[thick,bend right,->](-.9,2.5)to(0,2.3);
    \draw (-.7,1.5) node {\tiny$[\za\zb]$};
    \draw[thick,bend right,->](0,2.3)to(.9,2.5);
    \draw (.7,1.5) node {\tiny$[\zg\delta]$};
    \draw[dark-green,thick](V1)to(V6);
    \draw (-2,-3.5) coordinate (V70) node {$\rpoint$};
    \draw (2,-3.5) coordinate (V50) node {$\rpoint$};
    \draw (0,-5.5) node {(B) $~~~(S_{L},M_L,\zD_L)$ corresponding to $A(\zD_L)$};
\end{tikzpicture}
\begin{tikzpicture}[scale=0.45]
	\draw (0,0) circle (4cm);
    \draw (0,-4) coordinate (V1) node {$\rpoint$};
    \draw (0,4) coordinate (V6) node {$\gpoint$};
    \draw (-4,0) coordinate (V9) node {$\rpoint$};
    \draw (4,0) coordinate (V3) node {$\rpoint$};
    \draw (3.5,-2) coordinate (V2) node {$\gpoint$};
    \draw (-3.5,-2) coordinate (V10) node {$\gpoint$};
    \draw (-2.2,-2.6) node {\tiny$\overline{\ell^*_2}$};
    \draw (2.2,-2.6) node {\tiny$\overline{\ell^*_4}$};
%

    \draw[red,thick](V1)to(V9);
    \draw[red,thick](V1)to(V3);

    \draw[thick,bend right,->](.5,-3.5)to(-.5,-3.5);
    \draw (0,-2.5) node {\tiny$(\zb\zg)^{op}$};
        \draw[red,thick](V1)to(V3);
        \draw[red,thick](V1)to(V9);
    \draw (0,-5.5) node {(C) $(S_{\zD^{*}\setminus L^{*}},M_{\zD^{*}\setminus L^{*}},L^{*})$};
    \draw (.3,-6.5) node {corresponding to $(A(L))^!$};
        \draw (6,0) node {};
\end{tikzpicture}
\begin{tikzpicture}[scale=0.45]
    \draw (-6,-.5) node {};
	\draw (0,0) circle (4cm);
    \draw (0,-4) coordinate (V1) node {$\rpoint$};
    \draw (0,4) coordinate (V6) node {$\gpoint$};
    \draw (-4,0) coordinate (V9) node {$\rpoint$};
    \draw (4,0) coordinate (V3) node {$\rpoint$};
    \draw (3.5,-2) coordinate (V2) node {$\gpoint$};
    \draw (-3.5,-2) coordinate (V10) node {$\gpoint$};
    \draw (-1.8,0) node {\tiny$\ell_2$};
    \draw (1.8,0) node {\tiny$\ell_4$};

    \draw[dark-green,thick](V10)to(V6);
    \draw[dark-green,thick](V2)to(V6);


    \draw[thick,bend right,->](-.5,3.2)to(.5,3.2);
    \draw (0,2.5) node {\tiny$\zb\zg$};
    \draw (0,-5.5) node {(D) $(S_{\zD^{*}\setminus L^{*}},M_{\zD^{*}\setminus L^{*}},L)$};
    \draw (.3,-6.5) node {corresponding to $A(L)$};
       \draw (-8,0) node {};
\end{tikzpicture}
\end{center}
\begin{center}
\caption{Cut surfaces induced by  $L=\{\ell_2,\ell_4\}$ in Example~\ref{example:recollement-cut-surface}. The surface in (B) is obtained by cutting the surface in (A) along $L$ and the surface in (C) is obtained by cutting the surface in (A) along $\{\ell_{1}^{*}, \ell_{3}^{*},\ell_{5}^{*}\}$.}\label{figure:recollement-cut-surface}.
\end{center}
\end{figure}
\end{Ex}

\begin{Rem}\label{remark:ribbongraph}
\begin{enumerate}[\rm (1)]
\item
     Given a subset $L$ of arcs of $\zD$, we can consider $L$ as what is called a ribbon graph which is graded with a grading  induced by the grading of $\zD$. From $L$ we can construct a graded marked surface dissection $(S(L), M(L), L)$, see for example \cite{HKK17, LP20}.

\item
It is straightforward to see that the graded algebra $A(L)$ associated to the graded sub-ribbon graph $L$ of $\zD$ is isomorphic to $eAe$, where $e \in A$ is the idempotent corresponding to $L$. Hence by Proposition~\ref{Prop:Koszul} and Theorem~\ref{main theorem cutting surfaces}, $A(L)$ is isomorphic to $(A({{\zD^*_{\zD^{*}\setminus L^{*}}}}))^!$.

\item

It is easy to see  that $L$ gives rise to an admissible dissection of   $(S_{\zD^{*}\setminus L^{*}},M_{\zD^{*}\setminus L^{*}})$ and the algebra associated to the surface dissection $L$ is exactly the algebra $A(L)$ associated to the graded ribbon graph $L$ in (2).
\end{enumerate}

\end{Rem}

Remark~\ref{remark:ribbongraph} implies the following.

\begin{Cor}
The marked surface $(S(L), M(L))$  obtained from the subribbon graph $L$ of $\zD$ is homeomorphic to the marked surface  $(S_{\zD^{*}\setminus L^{*}},M_{\zD^{*}\setminus L^{*}})$.
\end{Cor}

 We now give the precise definitions of the terms above and then devote the rest of the Section to proving the stated results. We begin by defining the surface obtained from cutting arcs on a marked surface with a graded admissible dissection.

\begin{Def}\cite{APS19,CS20}\label{definition:cut surface1}
Let $\ell$ be an $\gpoint$-arc without self-intersection in the interior of a marked surface $(S,M)$. Let $p$ and $q$ be the endpoints of $\ell$ which may coincide.
We define the \emph{cut surface} $(S_\ell , M_\ell)$ to be the marked surface obtained from $(S,M)$ by cutting  and  contracting along $\ell$. In particular, we set $M_\ell^{{\rpoint}}=M^{\rpoint}, P_\ell^{{\rpoint}}=P^{\rpoint}$, and $M_\ell^{{\gpoint}}\cup P_\ell^{{\gpoint}}=((M^{\gpoint}\cup P^{\gpoint})\setminus \{p,q\}) \cup \{\overline{p},\overline{q}\}$, where $\overline{p}$ and $\overline{q}$ are new $\gpoint$-marked points (on the boundary or in the interior) of $S_\ell$ obtained from cutting and contracting $\ell$.

We similarly define the cut surface $(S_{\ell^*} , M_{\ell^*})$ of $(S,M)$ by cutting along an $\rpoint$-arc $\ell^*$ which has no self-intersection in the interior of the surface.
\end{Def}

 Let $\ell'$ be an $\gpoint$-arc on $(S,M)$, denote by $\overline{\ell'}$ the \emph{induced arc} of $\ell'$ on $(S_\ell, M_\ell)$, see Figure \ref{figure:cut surface} for an example. Note that if $\ell'$ intersects with $\ell$, then $\overline{\ell'}$ is empty.

\begin{Ex}\label{example:cut surface}
Figure \ref{figure:cut surface} gives an example of a cut surface and induced arcs. Namely, we cut an annulus along an arc $\ell$ connecting the two boundary components and then contract along the cut. From this we see that the cut surface  is a disk.

\begin{figure}[ht]
\begin{center}
\begin{tikzpicture}[scale=0.4]
\begin{scope}
	\draw (0,0) circle (1cm);
	\clip[draw] (0,0) circle (1cm);
	\foreach \x in {-4,-3.5,-3,-2.5,-2,-1.5,-1,-0.5,0,0.5,1,1.5,2,2.5,3,3.5,4}	\draw[xshift=\x cm]  (-5,5)--(5,-5);
\end{scope}
	\draw (0,0) circle (4cm);
    \draw (0,-1) node {$\gpoint$};
    \draw (0,-4) node {$\gpoint$};
    \draw (0,4) node {$\gpoint$};
    \draw (-4,0) node {$\gpoint$};
    \draw (4,0) node {$\gpoint$};

    \draw (0,1) node {$\rpoint$};
    \draw (-3.5,2) node {$\rpoint$};
    \draw (3.5,2) node {$\rpoint$};
    \draw (-2,-3.5) node {$\rpoint$};
    \draw (2,-3.5) node {$\rpoint$};

    \draw (0.4,-2.5) node {\tiny$\ell$};
    \draw[blue,thick](0,-4)to(0,-1);
%

    \draw (-.4,-4.5) node {\tiny$p$};
    \draw (-.4,-1.5) node {\tiny$q$};
    \draw (-1.8,1) node {\tiny$\ell'$};
    \draw (1.8,1) node {\tiny$\ell''$};
    \draw[dark-green,thick](0,-1)..controls (-1.8,-1) and (-1.8,1.3)..(0,4);
    \draw[dark-green,thick](0,-1)..controls (1.8,-1) and (1.8,1.3)..(0,4);
    \draw (0,-5.5) node {\tiny$(S,M)$};
\end{tikzpicture}
\begin{tikzpicture}[scale=0.4]
\begin{scope}
	\draw (0,0) circle (1cm);
	\clip[draw] (0,0) circle (1cm);
	\foreach \x in {-4,-3.5,-3,-2.5,-2,-1.5,-1,-0.5,0,0.5,1,1.5,2,2.5,3,3.5,4}	\draw[xshift=\x cm]  (-5,5)--(5,-5);
\end{scope}
	\draw (0,0) circle (4cm);
    \draw[dark-green,thick,fill=white] (-.185,-.97) circle (0.15);
    \draw[dark-green,thick,fill=white] (.185,-.97) circle (0.15);
    \draw[dark-green,thick,fill=white] (-.185,-3.97) circle (0.15);
    \draw[dark-green,thick,fill=white] (.185,-3.97) circle (0.15);

    \draw (0,4) node {$\gpoint$};
    \draw (-4,0) node {$\gpoint$};
    \draw (4,0) node {$\gpoint$};

    \draw (0,1) node {$\rpoint$};
    \draw (-3.5,2) node {$\rpoint$};
    \draw (3.5,2) node {$\rpoint$};
    \draw (-2,-3.5) node {$\rpoint$};
    \draw (2,-3.5) node {$\rpoint$};

    \draw[blue,thick,dashed](-.2,-4)to(-.2,-1);
    \draw[blue,thick,dashed](.2,-4)to(.2,-1);

    \draw (-.5,-4.5) node {\tiny$p$};
    \draw (-.5,-1.4) node {\tiny$q$};
    \draw (.6,-4.4) node {\tiny$p$};
    \draw (.6,-1.3) node {\tiny$q$};

    \draw (-6.4,-5.8) node {};
\end{tikzpicture}
\begin{tikzpicture}[scale=0.4]
    \draw (-6,-.5) node {};
	\draw (0,0) circle (4cm);
    \draw (0,-4) node {$\rpoint$};
    \draw (0,4) node {$\gpoint$};
    \draw (-4,0) node {$\rpoint$};
    \draw (4,0) node {$\rpoint$};
    \draw (-2,3.5) node {$\rpoint$};
    \draw (2,3.5) node {$\rpoint$};

    \draw (3.5,-2) node {$\gpoint$};
    \draw (3.5,2) node {$\gpoint$};
    \draw (-3.5,-2) node {$\gpoint$};
    \draw (-3.5,2) node {$\gpoint$};

    \draw (-4.5,-2) node {\tiny$\overline{p}$};
    \draw (4.5,-2) node {\tiny$\overline{q}$};
%
%
%
    \draw[dark-green,thick](-3.5,-2)to(0,4);
    \draw[dark-green,thick](3.5,-2)to(0,4);
    \draw (-2.5,1) node {\tiny$\overline{\ell'}$};
    \draw (2.5,1) node {\tiny$\overline{\ell''}$};
    \draw (0,-5.5) node {\tiny$(S_\zg,M_\zg)$};
\end{tikzpicture}
\end{center}
\begin{center}
\caption{An example of a cut surface, where the labels $\overline{p}$ and $\overline{q}$ are allowed to be exchanged. The arcs $\overline{\ell'}$ and $\overline{\ell''}$ are respectively the induced arcs of $\ell'$ and $\ell''$.}\label{figure:cut surface}
\end{center}
\end{figure}
\end{Ex}

Note that the cut surface may not be connected, even if the original surface is. Recall that by our convention if the cut surface has a component corresponding to  a sphere with a unique $\gpoint$-puncture and a unique $\rpoint$-puncture or an unpunctured disk with only one $\gpoint$-marked point and one $\rpoint$-marked point on the boundary, then this corresponds to the zero gentle algebra and  we  can ignore this component.

\begin{Prop}\label{lemma:induce arrows}
Let $(S,M,\zD, G)$ be a  graded marked surface with associated  graded gentle algebra $A(\zD)$. Let $\ell$ be an arc in $\zD$ and denote  the associated idempotent in $A$ also by $\ell$. Denote by $\zD_\ell=\{\overline{\ell_i}, \ell_i\in \zD\}$ the set of arcs on the cut surface $(S_\ell,M_\ell)$  induced from the arcs in $\zD$.
\begin{enumerate}[\rm(1)]
\item
  Then $\zD_\ell$ is an admissible dissection on $(S_\ell,M_\ell)$.
\item The algebra $A(\zD_\ell)$ associated to $\zD_\ell$ is isomorphic to $A(\zD)_{\ell}$ as ungraded algebras, where $A(\zD)_{\ell}$ is the (graded) algebra introduced in Definition \ref{definition:left algebra}.
\end{enumerate}
\end{Prop}

With the notation of the above proposition we make the following definitions.

\begin{Def}
We call $\zD_\ell$ the \emph{induced admissible $\gpoint$-dissection of $\zD$} on
 $(S_\ell,M_\ell)$ and we call
 $(S_{\ell}, M_{\ell},\zD_\ell,G_\ell)$ the \emph{graded cut surface} of
 $(S, M, \zD,G)$ at $\ell$, where $G_\ell$ is the grading on $\zD_\ell$ such that $A(\zD_\ell)\cong A(\zD)_{\ell}$ as graded algebras.
\end{Def}

 \begin{proof}[Proof of Proposition~\ref{lemma:induce arrows}]
(1) Follows from \cite{APS19}.

(2)
Let $(Q, I)$ be such that $A(\zD)=kQ/I$ and  let  $A(\zD_\ell)=kQ'/I'$.
Since $A(\zD)$ is a gentle algebra, $A(\zD)_\ell=kQ_\ell/I_\ell$ is as defined in Example~\ref{Ex:gentleA_e}.

We now show that $kQ_\ell/I_\ell\cong kQ'/I'$.
It follows from (1) that  there is a canonical bijection from the vertex set of $Q_\ell$ to the vertex set of $Q'$, which sends an arc $\ell_i\in \zD\backslash \{ \ell\}$ to an arc $\overline{\ell_i}\in\zD_\ell$.

Under this bijection, we will further show that the set of arrows of $Q_\ell$ coincides with the set of  arrows of $Q'$. For this, let $\overline{\za}: \overline{\ell_i} \rightarrow \overline{\ell_j}$ be an arrow of $Q'$ corresponding to an oriented intersection at a common endpoint $\overline{\varrho}$ of $\overline{\ell_i}$ and $\overline{\ell_j}$. We claim that $\overline{\za}$ must correspond to an arrow of the form as described in Example~\ref{Ex:gentleA_e}~(2).

Suppose $\overline{\varrho}\notin \{\overline{p}, \overline{q}\}$ where if $\ell$ has endpoints $p$ and $q$, $\overline{p}$ and $ \overline{q}$ are the induced marked points in $(S_\ell, M_\ell)$. Then $\overline{\varrho}$ can be lifted to a marked point on $S$, which we denoted by $\varrho$. Note that since the local configuration of $\varrho$ does not change when cutting the surface, $\overline{\za}$ can be lifted to a minimal oriented intersection in $(S,M,\zD)$ from $\ell_i$ to $\ell_j$ at $\varrho$, which gives rise to an arrow $\za:\ell_i\rightarrow\ell_j$ in $Q$. Thus in this case, $\overline{\za}$ corresponds to an arrow $[\za]$ in $Q_\ell$, which belongs to the first subset as stated in Example~\ref{Ex:gentleA_e}~(2).

If $\overline{\varrho}\in \{\overline{p}, \overline{q}\}$, then both $\ell_i$ and $\ell_j$ share an endpoint with $\ell$, and up to permutating the endpoints $p$ and $q$ of $\ell$, we have the following three possibilities of lifting $\overline{\ell_i}$ and $\overline{\ell_j}$ to arcs $\ell_i$ and $\ell_j$ on $(S,M)$.

(a) Both $\ell_i$ and $\ell_j$ share the same endpoint (for example $p$) with $\ell$ such that $\ell$ is not between $\ell_i$ and $\ell_j$, see the following leftmost and middle pictures. In this case, note that the minimal oriented intersection $\overline{\za}$ is induced from a minimal oriented intersection $\za$ on $(S,M,\zD)$, which gives rise to an arrow of $Q_\ell$. Thus in this case, $\overline{\za}$ also corresponds to an arrow $[\za]$ of $Q_\ell$ in the first subset as stated in Example~\ref{Ex:gentleA_e}~(2).
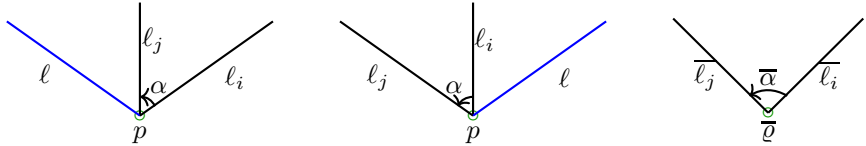
\[
 \begin{tikzpicture}[scale=0.25]
\draw (0,0) node {$\gpoint$};
\draw (1,1.5) node {$\za$};

\draw[->,thick,bend right](.8,0.5)to(0,1);
\draw[thick](0,0)--(0,6);
\draw[blue,thick](0,0)to(-7,5);
\draw[thick](0,0)to(7,5);
\draw (.7,4) node {$\ell_j$};
\draw (0,-1) node {$p$};
\draw (-5,2.2) node {$\ell$};
\draw (5,2) node {$\ell_i$};
\draw (10,2) node {};
\end{tikzpicture}
 \begin{tikzpicture}[scale=0.25]
\draw (0,0) node {$\gpoint$};
\draw (-1,1.5) node {$\za$};

\draw[<-,thick,bend left](-.8,0.5)to(0,1);

\draw[thick](0,0)--(0,6);
\draw[thick](0,0)to(-7,5);
\draw[blue,thick](0,0)to(7,5);
\draw (.6,4) node {$\ell_i$};
\draw (0,-1) node {$p$};
\draw (-5,2) node {$\ell_j$};
\draw (4.8,2) node {$\ell$};
\draw (10,2) node {};
\end{tikzpicture}
 \begin{tikzpicture}[scale=0.25]
\draw (0,0) node {$\gpoint$};
\draw (0,1.9) node {$\overline{\za}$};

\draw[<-,thick,bend left](-1,0.9)to(1,.9);

\draw[thick](0,0)to(-5,5);
\draw[thick](0,0)to(5,5);
\draw (-3.3,2) node {$\overline{\ell_j}$};
\draw (0,-1) node {$\overline{\varrho}$};
\draw (3.2,2) node {$\overline{\ell_i}$};
\end{tikzpicture}
\]

(b) Both $\ell_i$ and $\ell_j$ share the same endpoint $p$ with $\ell$ such that $\ell$ is between $\ell_i$ and $\ell_j$. Note that in this case, since $\overline{\za}$ arises from the intersection $\overline{\varrho}$, the other endpoint $q$ of $\ell$  must be a $\gpoint$-puncture with $\ell$  the unique arc in $\zD$ adjacent to it. That is, the local configuration must be as in the left picture of the following figure, with arrows $\za_1: \ell_i \rightarrow \ell$, $\za_3: \ell \rightarrow \ell_j$ and $\za_2: \ell \rightarrow \ell$ in $Q$ arising from minimal oriented intersections at $p$ and $q$ respectively.
Then $\overline{\za}$ corresponds to an arrow $[\za_1\za_2\za_3]$ of $Q_\ell$ as in the third subset of Example~\ref{Ex:gentleA_e}~(2).

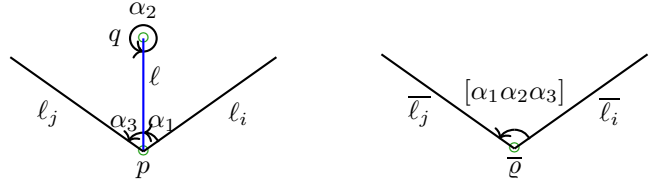
\[
 \begin{tikzpicture}[scale=0.25]
\draw (0,6) node {$\gpoint$};
\draw (0,0) node {$\gpoint$};

\draw (1,1.5) node {$\za_1$};

\draw (-1,1.5) node {$\za_3$};
%
\draw[->,thick,bend right](.8,0.5)to(0,1);
\draw[<-,thick,bend left](-.8,0.5)to(0,1);

\draw [thick] (0,6) circle (.7cm);
\draw[->,thick,bend right](-.2,5.5)to(0,5.3);

\draw[blue,thick](0,0)--(0,6);
\draw[thick](0,0)to(-7,5);
\draw[thick](0,0)to(7,5);
\draw (.6,4) node {$\ell$};
\draw (0,7.5) node {$\za_2$};
\draw (-1.5,6) node {$q$};
\draw (0,-1) node {$p$};
\draw (-5,2) node {$\ell_j$};
\draw (5,2) node {$\ell_i$};
\draw (12,2) node {};
\end{tikzpicture}
 \begin{tikzpicture}[scale=0.25]
\draw (0,0) node {$\gpoint$};

\draw (0,3) node {$[\za_1\za_2\za_3]$};

\draw[-,thick,bend right](.8,0.5)to(0,1);
\draw[<-,thick,bend left](-.8,0.5)to(0,1);

\draw[thick](0,0)to(-7,5);
\draw[thick](0,0)to(7,5);
\draw (0,-1) node {$\overline{\varrho}$};
\draw (-5,2) node {$\overline{\ell_j}$};
\draw (5,2) node {$\overline{\ell_i}$};
\end{tikzpicture}
\]

(c) The arcs $\ell_i$ and $\ell_j$ share different endpoints with $\ell$, see the  left picture below, where $\za_1: \ell_i\rightarrow \ell$ and $\za_2: \ell\rightarrow \ell_j$ are arrows in $Q$ arising from minimal oriented intersections of $q$ and $p$ respectively. In this case we have $[\za_1\za_2]$ as in the right picture, which belongs to the second subset in Example~\ref{Ex:gentleA_e}~(2).
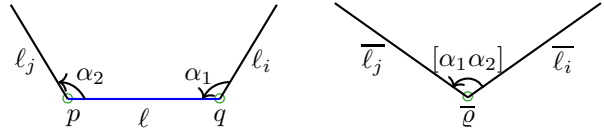
\[
\begin{tikzpicture}[scale=0.25]
\draw (-4,0) node {$\gpoint$};
\draw (4,0) node {$\gpoint$};
\draw[blue,thick](-4,0)--(4,0);
\draw[thick](-4,0)to(-7,5);
\draw[thick](4,0)to(7,5);

\draw[<-,thick,bend left](-4.5,0.9)to(-3.1,0);
\draw[->,thick,bend right](4.5,0.9)to(3.1,0);

\draw (2.9,1.2) node {$\za_1$};
\draw (-2.8,1.2) node {$\za_2$};

\draw (0,-1) node {$\ell$};
\draw (-3.7,-1) node {$p$};
\draw (4,-1) node {$q$};
\draw (-6.2,2) node {$\ell_j$};
\draw (6.2,2) node {$\ell_i$};
\end{tikzpicture}
\quad\quad
\begin{tikzpicture}[scale=0.25]
\draw (0,0) node {$\gpoint$};

\draw (0,2) node {$[\za_1\za_2]$};


\draw[-,thick,bend right](.8,0.5)to(0,1);
\draw[<-,thick,bend left](-.8,0.5)to(0,1);


\draw[thick](0,0)to(-7,5);
\draw[thick](0,0)to(7,5);
\draw (0,-1) node {$\overline{\varrho}$};
\draw (-5,2) node {$\overline{\ell_j}$};
\draw (5,2) node {$\overline{\ell_i}$};
\end{tikzpicture}
\]

We have thus established  a bijection between the arrow set of $Q_\ell$ and the arrow set of $Q'$, which induces an isomorphism of  the quivers $Q_\ell$ and $Q'$. Furthermore one can  check using the definition of a gentle algebra  that the ideal $I'$ of the algebra arising from the admissible dissection $\zD_\ell$ is generated exactly by  the relations in $I_\ell$ described in Example~\ref{Ex:gentleA_e}~(4).
\end{proof}

\begin{Lem}\label{Lem:ordered cut}
Let $(S,M,\zD,G)$ be a graded marked surface. Let $\ell$ and $\ell'$ be two arcs in $\zD$. Then there is a homeomorphism between $((S_\ell)_{\ell'},(M_\ell)_{\ell'})$ and $((S_{\ell'})_\ell,(M_{\ell'})_\ell)$, such that $((\zD_\ell)_{\ell'}, (G_\ell)_{\ell'})$ and $((\zD_{\ell'})_{\ell}, (G_{\ell'})_{\ell})$ coincide under this homeomorphism.
\end{Lem}
\begin{proof}
By Proposition \ref{Prop:e'e''} and Proposition \ref{proposition:red dual}, we have an isomorphism of graded gentle algebras associated to graded admissible $\gpoint$-dissections $((\zD_\ell)_{\ell'}, (G_\ell)_{\ell'})$ and $((\zD_{\ell'})_{\ell}, (G_{\ell'})_{\ell})$. Then the statement follows from the bijection between  marked surfaces with graded admissible dissections up to homeomorphism and  isomorphism classes of graded gentle algebras.
\end{proof}

We now define the \emph{graded cut surface} $(S_L,M_L, \zD_{L}, G_{L})$ of a graded marked surface $(S,M,\zD, G)$ at a subset $L\subset \zD$ to be the graded marked surface obtained by successively cutting $(S,M, \zD,G)$ along the arcs in $L$. Note that by
Lemma \ref{Lem:ordered cut}, the graded cut surface is independent up to homeomorphism of the order in which we cut the arcs in $L$.  Thus  it is well-defined. Similarly, one can cut along arcs  $D^* \subset \zD^*$ to obtain a graded $\rpoint$-surface dissection $\zD^*_{D^*}$ of the graded cut surface $(S_{D^*}, M_{D^*})$.

  The following observation naturally generalizes Proposition \ref{lemma:induce arrows}.

\begin{Prop}\label{proposition:red dual}
Let $(\zD,G)$ be a graded $\gpoint$-admissible dissection on $(S,M)$ with associated  graded gentle algebra $A$. Let $L$ be a subset of $\zD$ corresponding to an idempotent $L\in A$.
 Then we have isomorphism of graded gentle algebras $A(\zD_L)\cong A(\zD)_{L}$.
\end{Prop}

We now prove Theorem~\ref{main theorem cutting surfaces}

\begin{proof}[Proof of Theorem~\ref{main theorem cutting surfaces}]
 The left hand side of the theorem follows from Theorem \ref{Thm:recollement1} and Proposition~\ref{proposition:red dual}. For the right hand side,
 let $(S,M)$ be a marked surface with a graded admissible $\gpoint$-dissection $\zD$.  Denote by $A = A(\zD)$ the graded gentle algebra associated to $\zD$. Then  it follows from Proposition \ref{Prop:Koszul} that  $eAe \simeq ((A^!)_{1-e})^!$ where $e$ is the idempotent in $A$ corresponding to the arcs in $L$. By our construction $ (A^!)_{1-e}  \simeq A(L^*)  \simeq (A(L))^!  $ and thus $A(L) \simeq ((A^!)_{1-e})^!$. Hence   $A(L) \simeq eAe$. Thus by  Theorem \ref{Thm:recollement1}  the right hand side of the recollement is $A(L)$.
\end{proof}

\section{Application to partially wrapped Fukaya categories}

In this section we apply our results to the partially wrapped Fukaya category of a graded surface with stops as defined in \cite{HKK17}. More precisely,   if a graded gentle algebra $A$ with associated graded marked surface $(S,M,\zD)$ is homologically smooth and proper then $\D^{\rm b}(A)$ is triangle equivalent to the partially wrapped Fukaya category $\mathcal{W}(S,M,\eta(\zD))$, see \cite[Proposition 3.4]{HKK17} and \cite[Lemma 3.3]{LP20}, where $\eta(\zD)$ is the line field corresponding to the grading of $\zD$. In this case  we have the following  recollement of partially wrapped Fukaya categories as a consequence of Theorem~\ref{main theorem cutting surfaces}. For the statement of the next result we use Remark~\ref{remark:ribbongraph}.

\begin{Thm}\label{Thm:RecollementFukaya}
Let  $\mathcal{W}(S,M,\eta(\zD))$ be the partially wrapped Fukaya category associated to a graded surface dissection $(S, M, \zD)$  with induced line field $\eta(\zD)$.  Let $L \subset \zD$ be a subset of $\zD$ such that  the graded marked surface $\mathcal{W}(S(L),M(L),\eta(L))$ generated by $L$ (as graded ribbon subgraph of $\zD$) has no marked points in the interior. Then there is a recollement of partially wrapped Fukaya categories

\begin{equation}\label{eq:recollementFukaya}
\xymatrix @C=3pc{
\mathcal{W}(S_L,M_L,\eta(\zD_L)) \ar[r]^{}
& \mathcal{W}(S,M,\eta(\zD)) \ar@/_1.5pc/[l]_{} \ar@/^1.5pc/[l]^{}
 \ar[r]^{}
&\mathcal{W}(S(L),M(L),\eta(L))  \ar@/^-1.5pc/[l]_{} \ar@/^1.5pc/[l]^{}
  } \end{equation}
where the left side corresponds to the graded cut marked surface along $L$.
\end{Thm}

Before proving Theorem~\ref{Thm:RecollementFukaya}, we show the following which is a consequence of Theorem~\ref{Thm:recollement1}.

\begin{Prop}\label{Thm:Fukaya}
Let $A=kQ/I$ be a graded gentle algebra and let $e$ be an idempotent in $A$.
If two of the three algebras $A$, $eAe$ and $A_{e}$ are homologically smooth and proper then so is the third one. Moreover, we have an induced  recollement
\begin{equation} \label{inducedrecoll}
\xymatrixcolsep{4pc}\xymatrix{
\D^{\rm b}(A_{e}) \ar[r]^{i_{*}=i_{!}}
&\D^{\rm b}(A) \ar@/_1.5pc/[l]_{i^{*}} \ar@/^1.5pc/[l]^{i^{!}}
 \ar[r]^{?\ot_{A}^{\bf L}Ae}
&\D^{\rm b}(eAe) \ar@/^-1.5pc/[l]_{?\ot^{\bf L}_{eAe}eA} \ar@/^1.5pc/[l]^{\RshHom_{eAe}(Ae, ?)}
  }, \end{equation}
   such that $i^{*}$ sends $e_{j}A$ to $e_{j}A_{e}$ and $i_{*}$ sends $\Top(e_{j}A_{e})$ to $\Top (e_{j}A)$ for any $j\not=1,2,\dots, m$.
In this case, $\D^{\rm b}(eAe)$ is triangle equivalent to $\D^{\rm b}((A^{!})_{(1-e)})$.
\end{Prop}

We point out that if we only assume one of $A$, $A_{e}$ and $eAe$ is homologically smooth and proper, Proposition \ref{Thm:Fukaya} is not true in general. We give an example.
\begin{Ex}\label{Example counterexample}
Let $Q=\xymatrix{1 \ar@<.5pc>[r]^{\za} & 2  \ar@<.5pc>[l]^{\zb}}$ and $I=\langle \za\zb \rangle$. Let $A=KQ/I$ be the zero graded gentle algebra and let $e=e_{2}$. Then $eAe=K[X]/(X^{2})$ with degree $|X|=0$ and $A_{e}=K[Y]$ with $|Y|=-1$. The functor $?\ot_{eAe}^{\bf L}eA$ sends the simple $eAe$-module $k$ to the complex
\[ \cdots \xra{\zb\za} e_{2}A \xra{\zb\za} e_{2}A  \xra{\zb\za} e_{2}A \ra 0 \ra \cdots\]
which has infinite dimensional total cohomology. So $?\ot_{eAe}^{\bf L}eA$ does not restricts to $\D^{\rm b}$.
 \end{Ex}

For a homologically smooth and proper graded gentle algebra $A$, we show that the bounded derived categories of $A$ and its quadratic dual are equivalent.

\begin{Prop}\label{Prop:Extdual}
Let $A=kQ/I$ be a graded gentle algebra which is homologically smooth and proper.
Then there is a triangle equivalence $\D^{\rm b}(A) \simeq \D^{\rm b}(A^{!})$.
\end{Prop}
\begin{proof}
Let $(S,M,\Delta, G)$ be the graded surface associated to $A$. Let
$T:=\bop_{i\in Q_{0}}S_{i}$ where $S_{i}$ is the  simple dg $A$-modules corresponding to $i\in Q_{0}$. By \cite{HKK17},  $S_{i}$ corresponds to some curve $\zg_{i}$ on $S$. It is well-known that  in the ungraded case, the set $\Delta':=\{\zg_{i}\mid i\in Q_{0}\}$
is exactly the dual dissection $\Delta^{*}$  of $\Delta$ up to rotation, where $S_{i}$ can be realized as a complex of projective modules given by the curve $\zg_{i}$, which is exactly \eqref{resolofsimple} for the zero graded case. In general, when we take the grading into consideration, the arc $S_{i}$ does not change, and the total complex $\Tot(P)$ of \eqref{resolofsimple} is exactly the twisted object corresponding to $\zg_{i}$ constructed in \cite{HKK17}. So $\Delta'$, which is $\Delta^{*}$ up to rotation, is also an admissible dissection of $S$.

Now we consider the graded surface dissection $\zD'$ of $(S,M)$ whose grading is induced by $\eta(\zD)$.
Let $B$ be the graded gentle algebra corresponding to $(S, M, \Delta')$. Since by definition $\eta(\zD)$ is equal to $\eta(\zD')$, the partially wrapped Fukaya categories  $\cal W(S,M, \eta(\Delta))$ and $\cal W(S,M, \eta(\Delta'))$ are equal.  Then by \cite{HKK17} the following are equivalent
\[ \D^{\rm b}(A) \simeq \cal W(S,M, \eta(\Delta)) = \cal W(S,M, \eta(\Delta'))  \simeq \D^{\rm b}(B),  \]
where $S_{i}\in \D^{\rm b}(A)$ is sent to $\zg_{i}\in \cal W(S,M, \eta(\Delta))$, and in turn it is send to $e_{i}B\in \D^{\rm b}(B)$. Now let $\gamma = \bigoplus_{i \in Q_0} \gamma_i$ in $ \cal W(S,M, \eta(\Delta))$. Then   we obtain isomorphisms of graded algebras:
\[ B=\bop_{j\in \Z}\Hom_{\D^{\rm b}(B)}(B, B[j]) \cong \bop_{j\in \Z}\Hom_{\cal W(S,M, \eta(\Delta))}(\zg, \zg[j]) \cong \bop_{j\in \Z}\Hom_{\D^{\rm b}(A)}(T, T[j]) \cong A^{!},\]
where the last isomorphism is due to Proposition \ref{Prop:Koszuldual}.
Thus we have triangle equivalences $\D^{\rm b}(A) \simeq \D^{\rm b}(B) \simeq \D^{\rm b}(A^{!})$.
\end{proof}

 It follows from the proof of Proposition~\ref{Prop:Extdual} that the dual dissection gives rise to the geometric model for the quadratic dual. More precisely, we show the following.

\begin{Cor}
The surface $(S,M)$ together with the graded dual surface dissection $\zD^*$ is  the geometric model for $\D^{\rm b}(A^{!})$.
\end{Cor}

We now characterize when the algebras $A$, $A_e$ and $eAe$ are homologically smooth or proper, leading to the proof of  Proposition~\ref{Thm:Fukaya}.

 Recall, for example from  \cite[Proposition 3.5]{HKK17}, that $A$ is homologically smooth if and only if there are no cyclic paths $\za_{1}\cdots\za_{n}$ with $\za_{i}\za_{i+1}\in I$ for any $i\in\Z/n\Z$ and
 $A$ is proper if and only if there are no cyclic paths $\za_{1}\cdots\za_{n}$ with $\za_{i}\za_{i+1}\not\in I$ for any $i\in\Z/n\Z$. We have the following immediate consequence.

 \begin{Lem}\label{Lem:smoothvsproper}
 $A$ is homologically smooth (resp. proper) if and only if $A^{!}$ is proper (resp. homologically smooth).
 \end{Lem}

And we have the following observation.
\begin{Lem}
\begin{enumerate}[\rm(1)]
\item
If $A$ is proper, then $eAe$ is also proper.
\item If $A$ is homologically smooth, then $A_{e}$ is also homologically smooth.
\end{enumerate}
\end{Lem}
\begin{proof}
Since $eAe$ is a subalgebra of $A$, so it is clear that $eAe$ is proper showing (1) and (2) directly follows using Lemma~\ref{Lem:smoothvsproper}. \end{proof}

Notice that  $A$ is homologically smooth does not necessarily imply $eAe$ homologically smooth, see for example Example~\ref{Example counterexample}. The same example shows that  $A$ proper does not imply $A_e$ proper.

\begin{Lem}\label{Lem:AeAe}
Let $A$ be homologically smooth and proper. Then the following are equivalent.
 \begin{enumerate}[\rm (1)]
  \item $eAe$ is homologically smooth.
  \item $A_{e}$ is proper.
 \end{enumerate}
\end{Lem}

\begin{proof}
$(1)\Rightarrow (2)$ If $A_{e}$ is not proper, then  there is a cyclic  path $\za_{1}\cdots\za_{p}$ with $\za_{i}\za_{i+1}\not\in I$ for all $i\in \Z/p\Z$. We may assume without loss of generality that $\za_{l}=[\zb_{l_{1}}\cdots\zb_{l_{r(l)}}]$ for any $1\le l\le q$ and $\za_{l}=[\za_{l}]$ for any $q<l\le p$.
Then $\zb_{1_{1}}\cdots\zb_{1_{r(1)}}\cdots\zb_{q_{r(q)}}$ is a cyclic path in $eAe$ with any composition of successive paths is in $I$ and  $eAe$ is not homologically smooth.

$(2)\Rightarrow (1)$ Note that $A_{e}\cong ((1-e)A^{!}(1-e))^{!}$ and $eAe\cong ((A^{!})_{1-e})^{!}$ by Proposition \ref{Prop:Koszul}, and the result follows from the first part of the proof.
\end{proof}

Note that for $(1)\Rightarrow (2)$, it is enough to assume that $A$ is proper and dually, for $(2)\Rightarrow (1)$ it is enough to assume $A$ is homologically smooth.

\begin{Prop}\label{Prop:AetoA}
\begin{enumerate}[\rm(1)]
\item Assume $A_{e}$ and $eAe$ both are  proper, then $A$ is also proper.

\item Assume $A_{e}$ and $eAe$ both are homologically smooth, then $A$ is also homologically smooth.
\end{enumerate}
\end{Prop}
\begin{proof}
(1)
If $A$ is not proper, then there is a cyclic path $\za_{1}\cdots\za_{p}$ with $\za_{i}\za_{i+1}\not\in I$ for any $i\in\Z/p\Z$.  Recall that $e = e_1 + \cdots e_m$. If $\{s(\za_{1}), \dots, s(\za_{p})\} \cap \{1,\dots,m\}\not=\emptyset$, then $eAe$ contains a cycle without relations, which contradicts  the assumption that $eAe$ is proper. So $\{s(\za_{1}), \dots, s(\za_{p})\} \cap \{1,\dots,m\}=\emptyset$, and in this case, the cyclic path $\za_{1}\cdots\za_{p}$ is contained in $A_{e}$, which implies $A_{e}$ is non-proper, a contradiction.
Thus $A$ is proper.

(2) Since $A_{e}\cong ((1-e)A^{!}(1-e))^{!}$ and $eAe\cong ((A^{!})_{1-e})^{!}$ by Proposition \ref{Prop:Koszul}, we have that  (2)  directly follows from (1) and Lemma~\ref{Lem:smoothvsproper}.
\end{proof}

Now we  prove Proposition \ref{Thm:Fukaya}.

\begin{proof}[Proof of Proposition \ref{Thm:Fukaya}]
By Lemma \ref{Lem:AeAe} and Proposition \ref{Prop:AetoA}, we know that if  any two of $A$, $A_{e}$ and $eAe$ are homologically smooth and proper then  the third one is also homologically smooth and proper.
In this case, the recollement \eqref{mainrecoll} restricts to bounded derived categories by Proposition \ref{Prop:D^{b}}.
In this case, since $eAe\cong ((A^{!})_{1-e})^{!}$ by Proposition \ref{Prop:Koszul}, and since $eAe
$ is homologically smooth and proper, by Proposition \ref{Prop:Extdual}, we have a triangle equivalence $\D^{\rm b}(eAe)\simeq \D^{\rm b}((A^{!})_{(1-e)})$.
\end{proof}

We have the following geometric version of Proposition \ref{Thm:Fukaya}, recall that if $A$ is a graded gentle algebra associated to a graded marked surface $(S,M,\zD,G)$ then   $A$ is homologically smooth if and only if there are no $\rpoint$-punctures in $M$ and
$A$ is proper if and only if there are no $\gpoint$-punctures in $M$.

\begin{Cor}\label{Thm:geo-Fukaya}
Let $(S,M)$ be the marked surface with a graded admissible dissection $(\zD,G)$. Let $L$ be a subset of $\zD$. If two out of the three surfaces $(S,M)$, $(S_L,M_L)$ and $(S_{\zD^*\setminus L^*},M_{\zD^*\setminus L^*})$ have no punctures, then neither has  the third.
\end{Cor}

\begin{proof}[Proof of Theorem~\ref{Thm:RecollementFukaya}]
 Recall that if a graded gentle algebra $A$ with associated graded marked surface $(S,M,\zD)$ is homologically smooth and proper then $\D^{\rm b}(A)$ is triangle equivalent to the partially wrapped Fukaya category $\mathcal{W}(S,M,\eta_\zD)$, see \cite[Proposition 3.4]{HKK17} and \cite[Lemma 3.3]{LP20}. Then the result directly follows from Proposition~\ref{Thm:Fukaya} and the above.
\end{proof}

 We end this section by showing the following result.
\begin{Prop}\label{Prop:ladder}
 The recollement \eqref{inducedrecoll}  can be extended to an unbounded ladder.
 \end{Prop}
 Recall that an \emph{unbounded ladder} introduced in \cite[Section 3]{AKLY} is an infinite diagram of triangulated categories and triangle functors
 \begin{center}
\begin{picture}(120,60)

\put(0,20){\makebox(3,1){${\mathcal C}'$}}

\put(25,50){\vdots}

\put(10,39){\vector(1,0){30}}

\put(40,29){\vector(-1,0){30}}

\put(10,20){\vector(1,0){30}}

\put(40,11){\vector(-1,0){30}}

\put(10,1){\vector(1,0){30}}

\put(25,-10){\vdots}

\put(50,20){\makebox(3,1){${\mathcal C}$}}

\put(75,50){\vdots}

\put(60,39){\vector(1,0){30}}

\put(90,29){\vector(-1,0){30}}

\put(60,20){\vector(1,0){30}}

\put(90,11){\vector(-1,0){30}} \put(60,1){\vector(1,0){30}}
\put(75,-10){\vdots}

\put(100,20){\makebox(3,1){${\mathcal
C}''$}}

\put(23,43){\makebox(3,1){\scriptsize$i_{-2}$}}
\put(75,43){\makebox(3,1){\scriptsize$j_{-2}$}}
\put(75,33){\makebox(3,1){\scriptsize$i_{-1}$}}
\put(23,33){\makebox(3,1){\scriptsize$j_{-1}$}}
\put(23,24){\makebox(3,1){\scriptsize$i_0$}}
\put(75,24){\makebox(3,1){\scriptsize$j_0$}}
\put(23,15){\makebox(3,1){\scriptsize$j_1$}}
\put(75,15){\makebox(3,1){\scriptsize$i_1$}}
\put(23,5){\makebox(3,1){\scriptsize$i_2$}}
\put(75,5){\makebox(3,1){\scriptsize$j_2$}}
\put(275,15){\makebox(25,1)}
\end{picture}
\end{center}

\vskip10pt \noindent such that any three consecutive
rows form a recollement  of $\mathcal C$
relative to $\mathcal C'$ and $\mathcal C''$.
To show Proposition \ref{Prop:ladder}, we need the following two observations.
\begin{Lem}\label{Lem:Serre}
Assume $A$ is a dg $k$-algebra which is homologically smooth and proper, then $\D^{\rm b}(A)$ admits a Serre functor which is given by the Nakayama functor $\nu_{A}:=?\ot_{A}^{\bf L}DA$.
\end{Lem}
\begin{proof}
Since $A$ is homologically smooth and proper,
we have
$\per (A)=\D^{\rm b}(A)$. In particular, $\per (A)$ is Hom-finite and $\nu_{A}$ induces a triangle equivalence from $\per(A)$ to $\thick(DA)$ (see \cite[Section 10]{Keller94}).
Since we also have $\per(A^{\rm op})=\D^{\rm b}(A^{\rm op})$, by using the $k$-duality $D$, we obtain that $\thick(DA)=\D^{\rm b}(A)$. So $\nu_{A}$ gives an auto-equivalence of $\D^{\rm b}(A)$.
Moreover, for any $X\in \D(A)$ and $Y\in \per(A)$, we have
\[ \Hom_{\D(A)}(X, \nu_A(Y)) \cong D\Hom_{\D(A)}(Y,X). \]
Thus $\nu_{A}$ is a Serre functor of $\D^{\rm b}(A)$.
\end{proof}

\begin{Lem}\label{Lem:ZZZZ}
\cite[Lemma 2.8]{ZZZZ}
Let $\cal C$ and $\cal D$ be two Hom-finite triangulated categories with Serre functors $\nu_{\cal C}$ and $\nu_{\cal D}$ respectively.  Assume $F:\cal C\ra \cal D$ is a triangle functor with a right adjoint $G: \cal D\ra \cal C$. Then $F$ admits a left adjoint $\nu_{\cal C}^{-1}G\nu_{\cal D}$ and $G$ admits a right adjoint $\nu_{\cal D}F\nu_{\cal C}^{-1}$.
\end{Lem}

Now Proposition \ref{Prop:ladder} is immediately from above lemmas.
\begin{proof}[Proof of Proposition \ref{Prop:ladder}]
 Since $A$, $eAe$ and $A_{e}$ are all homologically smooth and proper, so their bounded derived categories admit Serre functors $\nu_{A}$, $\nu_{eAe}$ and $\nu_{A_{e}}$ respectively by Lemma \ref{Lem:Serre}. Since we already have  a recollement \eqref{inducedrecoll}, by Lemma \ref{Lem:ZZZZ}, we can inductively extend it to obtain infinitely many adjoints both upwards and downwards, which gives rise to a ladder by Proposition \ref{Prop:TTF} and  Lemma \ref{Lem:adjrecol}.
\end{proof}

\section{Exceptional sequences, silting and simple-minded collections for graded gentle algebras}

\subsection{ Silting reductions and simple-minded collection reductions}
In this section we use Theorems \ref{main theorem cutting surfaces} and \ref{Thm:RecollementFukaya} to study silting and simple-minded collection (SMC) reductions.
Recall that an object $P$ of a $k$-linear triangulated category $\T$ is called \emph{pre-silting} if $\Hom_{\T}(P, P[>\hs 0])=0$. We call $P$   \emph{silting}  if it is pre-silting and $\T =\thick_{\T}(P)$. Let $T=\bop_{i=1}^{m}T_{i}$ be a direct sum of indecomposable objects in $\T$. We say $T$ is a \emph{pre-simple-minded-collection} (or \emph{pre-SMC} for short), if  $\Hom_{\T}(T_{i},T_{j})=\delta_{i,j}k$  and $\Hom_{\T}(T_{i},T_{j}[<\hs0])=0$ for all $i, j$. We call $T$ a \emph{simple-minded collection} (or \emph{SMC}) if it is a pre-SMC and $\T=\thick_{\T}(T)$. Let $X\in\T$. We denote by $\Filt(X)$ the smallest extension-closed subcategory of $\T$ containing $X$. Note that if $X$ is pre-silting, then $\Filt(X)$ coincides with $\add X$, the smallest subcategory of $\T$ containing $X$ and closed under direct summands. We write ${}^{\perp_{\T}}X=\{Y\in\T \mid \Hom_{\T}(Y, X)=0\}$ and $X^{\perp_{\T}}=\{Y\in\T \mid \Hom_{\T}(X, Y)=0\}$ (or just write ${}^{\perp}X$ and $X^{\perp}$ for simplicity).

Let $A=kQ/I$ be a graded gentle algebra.
Let $P$ be a pre-silting object of $\per(A)$.
\emph{Silting reduction}, introduced by \cite{IY},
is defined as the Verdier quotient $\U=\per(A)/\thick(P)$.
 Under the conditions,
\begin{enumerate}
\item[(P1)] $\add P$ is covariantly finite in ${}^{\perp_{\per (A)}}P[>\hs 0]$ and contravariantly finite in $P[<\hs 0]^{\perp_{\per(A)}}$,
\item[(P2)] For any $X\in \per (A)$, we have $\Hom_{\per(A)}(X, P[l])=0=\Hom_{\per(A)}(P, X[l])$ for $l\gg 0$,
\end{enumerate}
we can realize $\U$ as a certain subfactor category $\frac{\cal Z}{[\add P]}$ of $\per(A)$, where $\cal Z:=({}^{\perp}P[>\hs 0])\cap (P[<\hs 0]^{\perp})$ (see \cite[Theorem 3.1]{IY}), and moreover,
there is a bijection between  silting objects in $\U$ and those in $\per(A)$ containing $P$ (\cite[Theorem 3.7]{IY}).

Let $T$ be a pre-SMC of $\D^{\rm b}(A)$, the \emph{SMC reduction} of $\D^{\rm b}(A)$ with respect to $T$ is defined as $\cal V=\D^{\rm b}(A)/\thick(T)$ (see \cite{J}). Similarly to silting reduction, under the conditions,
\begin{enumerate}
\item[(R1)] $\Filt(T)$ is covariantly finite in $T[>\hs 0]^{\perp_{\D^{\rm b}(A)}}$ and contravariantly finite in ${}^{\perp_{\D^{\rm b}(A)}}T[<\hs 0]$,
\item[(R2)] For any $X\in \D^{\rm b}(A)$, we have $\Hom_{\D^{\rm b}(A)}(X, T[l])=0=\Hom_{\D^{\rm b}(A)}(T, X[l])$ for $l\ll 0$,
\end{enumerate}
we can realize $\cal V$ as a certain subcategory $\cal S:=({}^{\perp}T[<\hs 0])\cap (T[>\hs 0]^{\perp})$ of $\D^{\rm b}(A)$ and
there is a bijection between SMCs in $\cal V$ and SMCs in $\D^{\rm b}(A)$ containing $T$ (\cite[Theorem 3.1]{J}).

It is known that reductions can be realized as the perfect derived categories of dg categories up to direct summands (\cite[Theorem 3.8]{Keller06}). But in general, it is not easy to describe these dg categories clearly. However, in the case of graded gentle algebras we have the following result which directly follows from Theorem \ref{main theorem cutting surfaces}.

\begin{Thm}\label{Thm:siltreduc}
 Let $A$ be a graded gentle algebra  corresponding to a graded surface dissection  $(S, M, \Delta)$ and let $e:=\sum_{i=1}^{m}e_{i}$. Let $P=eA$  be the corresponding projective dg $A$-module and let $T=\Top(eA)$ be the corresponding semi-simple dg $A$-module.
\begin{enumerate}[\rm (1)]
\item Assume $A$ is proper and $P$ is  pre-silting, then the silting reduction $\per(A) /\thick(P)$ is equivalent to $\per(A_{e})$ . Moreover, it is obtained by cutting $(S,M)$ along the arcs corresponding to $e$, and there is a bijection
\[ \{\mbox{silting objects in $\per(A_{e})$}\} \overset{1:1}{\longleftrightarrow} \{ \mbox{silting objects in $\per(A)$ containing $P$}\}.\]
\item Assume $A$ and $(1-e)A(1-e)$ are homologically smooth and proper, and  $T$ is a pre-simple minded collection of $\D^{\rm b}(A)$, then the simple minded collection reduction $\D^{\rm b}(A)/\thick(T)$ is equivalent to $\D^{\rm b}((1-e)A(1-e))\simeq \D^{\rm b}((A^{!})_{e})$.  Moreover,  it is obtained by cutting $(S,M)$ along the arcs corresponding to $e$ in the dual dissection  and there is a bijection
\[ \{\mbox{SMCs in $\D^{\rm b}((A^{!})_{e})$}\} \overset{1:1}{\longleftrightarrow} \{ \mbox{SMCs in $\D^{\rm b}(A)$ containing $T$}\}.\]
\end{enumerate}
\end{Thm}
\begin{Rem}
Let  $A$ be a  homologically smooth and proper graded gentle algebra corresponding to a graded surface dissection $(S,M,\Delta)$. Let $\zG$ be a collection of pairwise non-intersecting $\gpoint$-arcs in $(S,M)$ corresponding  to an object $X\in \D^{\rm b}(A)=\per A$. If $X$ is pre-silting, then the silting reduction is also given by cutting the surface along $\zG$. In this case, $X$ does  not necessarily corresponds to a projective dg $A$-module.
\end{Rem}

The following two lemmas guarantee conditions (P1), (P2) and (R1), (R2) mentioned above hold, which are necessary for us.
\begin{Lem}\label{Lem:P2R2}
Let $A$ be a proper dg $k$-algebra. For any $X,Y\in \per(A)$, we have $\Hom_{\per(A)}(X, Y[l])=0=\Hom_{\per(A)}(Y, X[l])$ for $l\gg0$ and for $l\ll 0$.
\end{Lem}
\begin{proof}
Note that we have $\Hom_{\per(A)}(A,A[i])=\h^{i}(A)$. Since $A$ is proper, the assertion holds for $X=Y=A$. So it also holds for any $X, Y\in \thick(A)=\per (A)$.
\end{proof}

\begin{Lem}\label{Lem:R1}
Let $A$ be a graded gentle algebra which is  homologically smooth and proper and let $e:=\sum_{i=1}^{m}e_{i}$. Let $T=\Top(eA)$ be the corresponding semi-simple dg $A$-module.  Assume $T$ is a pre-SMC in $\D^{\rm b}(A)$ and $(1-e)A(1-e)$ is homologically smooth. Then $\Filt(T)$ is functorially finite in $\D^{\rm b}(A)$.
\end{Lem}
\begin{proof}
Let $e=e_{1}+\cdots+e_{m}$ and let $e'=1-e$. Since $e'Ae'$ is proper   and by our assumption, it is  homologically smooth, we have that $A_{e'}$ is also homologically smooth and proper by Proposition \ref{Thm:Fukaya}.
 Then we have a recollement by Proposition \ref{Thm:Fukaya}
\begin{equation}
\label{functoria}
\xymatrixcolsep{4pc}\xymatrix{
\D^{\rm b}(A_{e'}) \ar[r]^{i_{*}=i_{!}}
&\D^{\rm b}(A) \ar@/_1.5pc/[l]_{i^{*}} \ar@/^1.5pc/[l]^{i^{!}}
 \ar[r]^{j^{*}=j^{!}}
&\D^{\rm b}(e'Ae') \ar@/^-1.5pc/[l]_{j_{!}} \ar@/^1.5pc/[l]^{j_{*}}
  }. \end{equation}
 Moreover, $i_{*}$ sends $S_{i}':=\Top(e_{i}A_{e'})$ to $S_{i}$ for $1\le i\le m$. So we have an equivalence $i_{*}: \D^{\rm b}(A_{e'})\ra \thick(T)$.
 By \eqref{functoria}, there are two $t$-structures
 \[\D^{\rm b}(A)=j_{!}(\D^{\rm b}(e'Ae'))\perp i_{*}(\D^{\rm b}(A_{e'}))= i_{*}(\D^{\rm b}(A_{e'}))\perp j_{*}(\D^{\rm b}(e'Ae')), \]
 which implies that $\thick(T)=i_{*}(\D^{\rm b}(A_{e'}))$ is functorially finite in $\D^{\rm b}(A)$.

 Next we show that $\Filt(T)$ is functorially finite in $\thick(T)$, then it is functorially finite in $\D^{\rm b}(A)$ and we are done.  Since $T$ is a pre-SMC in $\D^{\rm b}(A)$, it is a SMC in $\thick(T)$ and  moreover, $T':=\bop_{i=1}^{m}S_{i}'$ is a SMC in $\D^{\rm b}(A_{e'})$. Then by Proposition \ref{Prop:Koszuldual}, we have
 $(A_{e'})^{!}\cong \bop_{t\in \Z}\Hom_{\D^{\rm b}(A_{e'})}(T', T'[t])$ is a non-negative dg algebra with $\h^{0}((A_{e'})^{!})$ semisimple. So $A_{e'}$ is a non-positive graded gentle algebra. Thus $\Filt(T')$ is functorially finite in $\D^{\rm b}(A_{e'})$ by \cite[Proposition 3.11]{J} and morevoer, $\Filt(T)=i_{*}(\Filt(T'))$ is functorially finite in $\thick(T)$. So the assertion holds.
\end{proof}

Now we show Theorem \ref{Thm:siltreduc}.
\begin{proof}[Proof of Theorem \ref{Thm:siltreduc}]
(1) Let $e=e_{1}+\cdots+e_{m}$. Assume $P=eA$ is a pre-silting object in $\per(A)$. Then
 by Theorems \ref{Thm:Ne} and \ref{Thm:recollement1}, the functor $i^{*}$ in \eqref{mainrecoll} induces a triangle equivalence  $\overline{i^{*}}: \per(A)/j_{!}(\per(eAe)) =\per(A)/\thick(P) \simeq \per(A_{e})$ up to direct summands. Since $A$ is proper, $\per (A)$ is Hom-finite and moreover, $\add P$ is functorially finite in $\per (A)$. So (P1) holds. Note that (P2) follows from Lemma \ref{Lem:P2R2}. So we only need to show   $\per(A)/\thick(P)$ is idempotent complete, then $\overline{i^{*}}$ is an equivalence and the bijection is given by \cite[Theorem 3.7]{IY}.

 We know that $\per (A)/\thick(P)$ is equivalent to the subfactor category $\frac{\cal Z}{[\add P]}$ of $\per (A)$, where $\cal Z:=({}^{\perp}P[>\hs 0])\cap (P[<\hs 0]^{\perp})$ (\cite[Theorem 3.1]{IY}). Since $\per (A)$ is idempotent complete and Hom-finite, so is $\cal Z$. Since  it is well-known that any idempotent morphism in $\frac{\cal Z}{[\add P]}$ can be lifted to an idempotent morphism in $\cal Z$, which means $\frac{\cal Z}{[\add P]}$ is idempotent complete, we have $\per(A)/\thick(P)$ is idempotent complete. So the assertion is true.

 (2) Let $e'=1-e_{1}-\cdots-e_{m}$.  Assume $T=\bop_{i=1}^{m}S_{i}$ is a pre-SMC in $\D^{\rm b}(A)$.
 By Propositions \ref{Prop:Extdual} and \ref{Prop:Koszul}, and the recollement \eqref{functoria}, we have the following equivalences
\[  \D^{\rm b}(A)/\thick(T)=\D^{\rm b}(A)/i_{*}(\D^{\rm b}(A_{e'}))\xra{\simeq} \D^{\rm b}(e'Ae') \xra{\simeq} \D^{\rm b}((A^{!})_{e}).\]
Since by Lemmas \ref{Lem:P2R2} and \ref{Lem:R1}, (R1) and (R2) hold, the assertion follows from \cite[Theorem 3.1]{J}.
\end{proof}

\subsection{Existence of  full exceptional sequences }\label{Sec:Exceptional}
In this section, we study the existence of full exceptional sequences  in $\D^{\rm b}( A)$
 for a graded gentle algebra $A$, which  is homologically smooth and proper.
 Exceptional sequences originate in algebraic geometry \cite{GR, Ru} and are instrumental in the construction of the Fukaya-Seidel categories \cite{Seidel}.
Furthermore, exceptional sequences give rise to semi-orthogonal decompositions. Semi-orthogonal decompositions of gentle algebras are studied in \cite{KS22} in terms of surface cuts.
 We begin by recalling the definitions of exceptional sequences.
 An object $X$ in $\D^{\rm b}(A)$ is \emph{exceptional} if $\Hom_{\D^{\rm b}(A)}(X,X[\neq 0]) = 0$ and $\Hom_{\D^{\rm b}(A)}(X,X) = k$. We call a sequence $(X_1, \ldots, X_n)$ \emph{exceptional} if each $X_i$ is an exceptional object in $\D^{\rm b}(A) $ and if $\Hom_{\D^{\rm b}(A)}(X_j,X_i[\Z]) = 0$ for all $i < j$.  An exceptional sequence is \emph{full} if $\thick_{\D^{\rm b}(A)} ( \bigoplus X_i) = \D^{\rm b}(A)$.

 We now define a special family of graded gentle algebras. Namely,
let $Q^{(n)}$ be the following quiver
 \begin{equation}
 \label{equ:An}
\xymatrix{1 \ar@<1pc>[r]^{\za_{1}} \ar@<-1pc>[r]^{\zg_{1}}& 2 \ar[l]_{\zb_{1}} \ar[r]^{\zd_{1}}& 3 \ar@<1pc>[r]^{\za_{2}} \ar@<-1pc>[r]^{\zg_{2}} & 4 \ar[l]_{\zb_{2}} \ar[r]^{\zd_{2}} & \cdots \ar[r]^{\zd_{n-1}} & 2n-1 \ar@<1pc>[r]^{\za_{n}}  \ar@<-1pc>[r]^{\zg_{n}}
& 2n \ar[l]_{\zb_{n}}}
\end{equation}
and let $I^{(n)}$ be the ideal generated by the set $\{\za_{i}\zb_{i}, \zb_{i}\zg_{i}, \zg_{j}\zd_{j}, \zd_{j}\za_{j+1} \mid 1\le i\le n, 1\le j\le n-1\}$.
Let $A^{(n)}:=kQ^{(n)}/I^{(n)}$ be the graded gentle algebra, whose grading is induced by a map $|\cdot |_{A^{(n)}}: Q^{(n)}_{1}\ra \Z$ (we will use $|\cdot|$ for simplicity). Write $a_{i}:=|\za_{i}|+|\zb_{i}|$ and $b_{i}:=|\zb_{i}|+|\zg_{i}|$ for any $1\le i\le n$. The marked surface associate to $A^{(n)}$ is a surface of genus $n$ with exactly one boundary component with one marked point. So all curves in this surface correspond to a loop or a closed curve.

Note that in this section, whenever we say that two graded algebras $A$ and $B$ are derived equivalent, we mean that  $\D^{\rm b}(A)$ and $\D^{\rm b}(B)$ are equivalent. We need the following observation.

 \begin{Lem}\label{Lem:g11}
 Let $A=kQ/I$ be a graded gentle algebra which is homologically smooth and proper. Let $(S, M, \Delta)$ be the graded marked surface associated with $A$. Let $\ell_{i}\in \Delta$  be the arc corresponding to vertex $i\in Q_{0}$.
 Then the following are equivalent.
 \begin{enumerate}[\rm (1)]
 \item $e_{i}Ae_{i}\not =k$ for any $i\in Q_{0}$,
 \item $\ell_{i}$ is a loop for any $i\in Q_{0}$,
 \item $(S,M)$ has only one boundary and one marked $\gpoint$-point.
 \end{enumerate}
 Moreover, if  one of the three holds then   $A$ is derived equivalent to a graded gentle algebra of the form $A^{(n)}$, for some $n\ge 1$. \end{Lem}
  \begin{proof}
 The equivalence of (1)-(3) directly follows from the definitions. Now we assume that (1)-(3) hold. Note $(S,M)$ has a grading as a surface where the grading is induced by the graded gentle algebra $A$, see \cite{LP20}. Since $A$ is homologically smooth and proper, $\D^{\rm b}(A)$ is derived equivalent to the partially wrapped Fukaya category associated to the graded marked surface $(S,M)$, \cite{HKK17, LP20}. Then since the surface associated to an algebra of the form $A^{(n)}$ is as in (3), by \cite{HKK17} there exists a grading on $A^{(n)}$ such that $\D^{\rm b}(A^{(n)})$ and $\D^{\rm b}(A)$ are equivalent  since both are equivalent to the same partially wrapped Fukaya category associated to the graded marked surface $(S,M)$.
 \end{proof}

 The following result is a graded version of   \cite[Theorem A]{CS22}, which follows from the description of exceptional dissections in a marked surface in \cite[Proposition 3.5]{CS22} and Proposition \ref{Thm:Fukaya}.

\begin{Thm}\label{Prop:excep}
Let $\mathcal{W}(S,M,\eta(\zD))$ be a partially wrapped Fukaya category associated to a graded surface dissection $(S,M,\eta(\zD))$.
 Then $\mathcal{W}(S,M,\eta(\zD))$ has a full exceptional sequence if and only if
 $(S,M)$ is not homeomorphic to a marked surface of genus $\ge 1$ with exactly one boundary and one marked point on that boundary.
\end{Thm}

Theorem \ref{Prop:excep} can be re-stated as follows.
\begin{Thm}\label{Prop:excep1}
 Let $A$ be a graded gentle algebra which is proper and homologically smooth, then
$\D^{\rm b}(A)$ has a full exceptional sequence if and only if  it is not derived equivalent to $A^{(n)}$ for some $n\ge 1$.
\end{Thm}
\begin{proof}
We first show the `if' part.
Let $(S,M, \Delta)$ be the graded marked surface associated to $A=kQ/I$. Since $A$ is not derived equivalent to $A^{(n)}$ for any grading of the arrows, then  by Lemma \ref{Lem:g11} there are at least two marked points on $S$ and thus by \cite[Proposition 3.5]{CS22}, there exists an exceptional  dissection $\Delta'$ of $S$. We may assume $\Delta=\Delta'$ for simplicity.

Let $Q_{0}=\{1, 2,\cdots, n\}$.
 Since $Q$ is a quiver without oriented cycles (see Definition \ref{Def:exceptionaldissection}),  we have $e_{1}Ae_{1}=k$ and by Proposition \ref{Thm:Fukaya}, there is a recollement
 \[\xymatrixcolsep{4pc}\xymatrix{
\D^{\rm b}(A_{e_{1}}) \ar[r]^{F}
&\D^{\rm b}(A) \ar@/_1.5pc/[l]_{F_{\lambda}} \ar@/^1.5pc/[l]_{F_{\sigma}}
 \ar[r]^{?\ot_{A}^{\bf L}Ae_{1}}
&\D^{\rm b}(k), \ar@/^-1.5pc/[l]_{?\ot^{\bf L}_{e_{1}Ae_{1}}e_{1}A} \ar@/^1.5pc/[l]
 } \]
where $A_{e_{1}}=kQ_{e_{1}}/\langle I_{e}\rangle$ with $(Q_{e_{1}})_{0}=\{2,3,\dots, n\}$ and
 $kQ_{e_{1}}$ has no oriented cycles (see Definition \ref{definition:left algebra}).
 Let $A_{0}:=A$ and $A_{1}:=A_{e_{1}}$. We inductively construct a series of recollements

  \[\xymatrixcolsep{4pc}\xymatrix{
\D^{\rm b}(A_{i}) \ar[r]
&\D^{\rm b}(A_{i-1}) \ar@/_1.5pc/[l] \ar@/^1.5pc/[l]_{{}}
 \ar[r]
&\D^{\rm b}(k). \ar@/^-1.5pc/[l] \ar@/^1.5pc/[l]
 } \]
 where $A_{i}:=(A_{i-1})_{e_{i}}$ for $1\le i\le n-1$. Note that $A_{n-1}=k$ by our construction. So we get a full exceptional sequence by \cite[Lemma 2.15]{Kalck}.

 For  the `only if' part, assume $A$ is derived equivalent to $A^{(n)}$ and $\D^{\rm b}(A)$ has a full exceptional sequence. Then $\D^{\rm b}(A^{(n)})$ also has a full exceptional sequence. In particular, there exists an indecomposable object $X$ such that $\Hom_{\D^{\rm b}(A^{(n)})}(X,X[i])=0$ for $i\not=0$ and $\Hom_{\D^{\rm b}(A^{(n)})}(X,X)=k$. Since the surface of $A^{(n)}$ only has one boundary and one marked point, by \cite{HKK17} the admissible curve corresponding to $X$ is a loop or a closed curve. It follows from  \cite{HKK17}
that $\dim \bop_{i\in\Z}\Hom_{\D^{\rm b}(A)}(X, X[i])\ge 2$. It is a contradiction. So $A$ is not derived equivalent to $A^{(n)}$.
 \end{proof}

A full exceptional sequence decomposes the derived category into derived categories of vector spaces. This decomposition in terms of the associated surface can be  seen as repeatedly cutting off discs where the surface dissection is given by a single arc without self-intersection.

\begin{Ex} An example of a  decomposition of $\D^{\rm b}(A)$ as in the proof of Theorem~\ref{Prop:excep}.
Let $Q: 1\xra{\za} 2 \xra{\zb} 3 \xra{\zg} 4$ and let $I=\lan\za\zb, \zb\zg \ran$. Let $A=kQ/I$ be a graded gentle algebra. Then $(P_{4}, P_{3}, P_{2}, P_{1})$ is a full exceptional sequence of $\D^{\rm b}(A)$. And we have the following decomposition of $\D^{\rm b}(A)$:
  \[\xymatrixcolsep{4pc}\xymatrix{
\D^{\rm b}(2\xra{\zb} 3 \xra{\zg} 4) \ar[r]
&\D^{\rm b}(1\xra{\za} 2\xra{\zb} 3 \xra{\zg} 4) \ar@/_1.5pc/[l]^{} \ar@/^1.5pc/[l]_{{}}
 \ar[r]
&\D^{\rm b}(k). \ar@/^-1.5pc/[l] \ar@/^1.5pc/[l]
 } \]
   \[\xymatrixcolsep{4pc}\xymatrix{
\D^{\rm b}(3 \xra{\zg} 4) \ar[r]
&\D^{\rm b}(2\xra{\zb} 3 \xra{\zg} 4) \ar@/_1.5pc/[l]^{} \ar@/^1.5pc/[l]_{{}}
 \ar[r]
&\D^{\rm b}(k). \ar@/^-1.5pc/[l] \ar@/^1.5pc/[l]
 } \]

   \[\xymatrixcolsep{4pc}\xymatrix{
\D^{\rm b}(k) \ar[r]
&\D^{\rm b}(3 \xra{\zg} 4) \ar@/_1.5pc/[l]^{} \ar@/^1.5pc/[l]_{{}}
 \ar[r]
&\D^{\rm b}(k). \ar@/^-1.5pc/[l] \ar@/^1.5pc/[l]
 } \]
 where the relations are given by $\za\zb$ and $\zb\zg$ and the grading of one arrow is the same as its grading in $Q$.

We note that any permutation of the projective indecomposables will give a decomposition of $\D^{\rm b}(A)$ into $\D^{\rm b}(k)$ induced by a full exceptional sequences in $\D^{\rm b}(A)$. Note that the full exceptional sequences in this case,  do not have only projective indecomposables as their terms.
\end{Ex}

\begin{Rem}
 \begin{enumerate}[\rm (1)]
 \item
Let $A$ be a homologically smooth and proper dg algebra (not necessarily gentle). Let $X$ be an exceptional object in $\D^{\rm b}(A)$. We may regard the Verdier quotient  $\cal V = \D^{\rm b}(A)/\thick(X)$ as the \emph{exceptional reduction} of $\D^{\rm b}(A)$ at $X$. Then
there are  bijections between the full exceptional sequences in $\cal V$ and  full exceptional sequences  in $\D^{\rm b}(A)$ ending at $X$, and those in $\D^{\rm b}(A)$ starting at $X$.
\item  If $A$ is a graded gentle algebra which is homologically smooth and proper, then the  exceptional reduction of $\D^{\rm b}(A)$
at $X$ is obtained by cutting the  surface associated to $A$ along the arc corresponding to $X$.
\end{enumerate}
\end{Rem}

\subsection{Existence of silting objects and simple-minded collections}
 Recall the graded gentle algebra $A^{(n)}$ constructed from  the quiver in \eqref{equ:An} in the previous section. Our main result in this section is as follows.

 \begin{Thm}\label{Thm:exist}
Let $\mathcal{W}(S,M,\eta(\zD))$ be a partially wrapped Fukaya category.
 Then $\mathcal{W}(S,M,\eta(\zD))$ has   a  silting object  if and only if it has a simple-minded collection and this is the case if
\begin{enumerate}[\rm (1)]
\item $\mathcal{W}(S,M,\eta(\zD))$ is not triangle equivalent to $\D^{\rm b}(A^{(n)})$, or
\item $\mathcal{W}(S,M,\eta(\zD))$ is  triangle equivalent to $\D^{\rm b}(A^{(n)})$ with $a_{i}\not=1$ or $b_{i}\not=1$ for any $1\le i\le n$.
\end{enumerate}
\end{Thm}
 We re-state Theorem \ref{Thm:exist} as  follows.

\begin{Thm}\label{Thm:exist1}
Let $A$ be a graded gentle algebra which is homologically smooth and proper. Then
 $\D^{\rm b}(A)$ has a silting object if and only if it has  a SMC and this is the case if
\begin{enumerate}[\rm (1)]
\item $A$ is not derived equivalent to $A^{(n)}$, or
\item $A$ is derived equivalent to $A^{(n)}$ and $a_{i}\not=1$ or $b_{i}\not=1$ holds for any $1\le i\le n$.
\end{enumerate}
\end{Thm}

 The following  shows  the  first part of Theorem \ref{Thm:exist1}.
\begin{Prop}\label{Prop:silting-SMC}
Let $A$ be a homologically smooth and proper dg $k$-algebra. Then $\D^{\rm b}(A)$ has silting objects if and only if it has SMCs.
\end{Prop}
\begin{proof}
If $\D^{\rm b}(A)$ has a silting object $P$, then we have a triangle equivalence $\RshHom_{A}(P,?): \D^{\rm b}(A) \xra{\simeq}  \per (\shEnd(P))$ (see for example, \cite[Theorem 3.8]{Keller06}), where $\shEnd(P)$ is the endomorphism dg $k$-algebra which is non-positive because $P$ is silting.
Since $A$ is homologically smooth and proper, then $\per (A)=\D^{\rm b}(A)$ by Lemma \ref{Lem:Db=per},  so $P$ is perfect and $\RshHom_{A}(P,?)$ commutes with infinite direct sum.
Moreover, $\RshHom_{A}(P,?)$ gives an equivalence $\D(A) \simeq \D(\shEnd(P))$ which restricts to  an equivalence $\D^{\rm b}(A) \simeq \D^{\rm b}(\shEnd(P))$ by Lemma \ref{Lem:D|per}.
 Thus we have $\per (\shEnd(P))=\D^{\rm b}(\shEnd(P))$. Since $\D^{\rm b}(\shEnd(P))$ has a SMC given by  simple dg modules,  $\D^{\rm b}(A)$ also has a SMC.

On the other hand, assume $\D^{\rm b}(A)$ has a SMC $X$. Let $P_{X} \ra X$ be a $\cal H$-projective resolution of $X$ and let $B:=\shEnd(P_{X})$ be the Koszul dual dg algebra (\cite[Section 10]{Keller94}). Then we have a triangle equivalence $\RshHom_{A}(P_{X}, ?): \D^{\rm b}(A) \simeq \per(B)$, which sends $X$ to the free dg $B$-module $B$. Since $\D^{\rm b}(A)=\per(A)$, we have that $P_{X}\in \per (A)$ and $\RshHom_{A}(R_{X}, ?)$ can be extended to  an equivalence $\D(A)\simeq \D(B)$,  which restricts to $\D^{\rm b}(A) \simeq \D^{\rm b}(B)$ by Lemma \ref{Lem:D|per}. So we have $\per (B)=\D^{\rm b}(B)$. Note that  $\h^{<0}(B)=0$ and $\h^{0}(B)$ is semi-simple,  so $\h^{0}(B)$ is a silting object in $\D^{\rm b}(B)$ by \cite[Corollary 4.1]{KN} and it gives a silting object in $\D^{\rm b}(A)$.
Thus the assertion holds.
\end{proof}

 Before proving Theorem \ref{Thm:exist1}, we need the following results.
 \begin{Lem}\label{Lem:ab}
 Let $A$ and $A'$ be two algebras of the form $A^{(1)}$ with different gradings. If $|\za|_{A}+|\zb|_{A}=|\za|_{A'}+|\zb|_{A'}$ and $|\zb|_{A}+|\zg|_{A}|=|\zb|_{A'}+|\zg|_{A'}$, then $\per(A)$ and $\per(A')$ are triangle equivalent.
 \end{Lem}

 \begin{proof}
 Let $P_{1}, P_{2}$ be the indecomposable  projective dg $A$-modules corresponding to the vertices $1$ and $2$, respectively. For any $i\in\Z$, the dg algebra $\shEnd_{A}(P_{1}\op P_{2}[i])$ is quasi-isomorphic to $\w{A}$ which as ungraded algebra is of the form $A^{(1)}$. Moreover, the grading  of $\w{A}$ is given by $|\za|_{\w{A}}=|\za|-i$, $|\zb|_{\w{A}}=|\zb|+i$ and $|\zg|_{\w{A}}=|\zg|-i$. Now let $i=|\za|_{A}-|\za|_{A'}$, then we have $\w{A}=A'$.
 So we have triangle equivalences $\per (A) \simeq\per (\shEnd_{A}(P_{1}\op P_{2}[i]))\simeq \per (A')$.
 \end{proof}

 \begin{Lem}\label{Lem:not=1}
 Let $A=A^{(1)}$.
 If $a_{1}\not=1$ or $b_{1}\not =1$, then $\D^{\rm b}(A)$ has a silting object.
 \end{Lem}
 \begin{proof}
 If $a_{1}\le 0$ or $b_{1}\le 0$ holds, we may assume $a_{1}\le 0$. Then by Theorem \ref{Thm:recollement1}, we have a recollement
 \[
\xymatrixcolsep{4pc}\xymatrix{
\D(A_{e_{2}}) \ar[r]^{i_{*}}
&\D (A) \ar@/_1.5pc/[l]_{i^{*}} \ar@/^1.5pc/[l]^{i^{!}}
 \ar[r]^{?\ot_{A}^{\bf L}Ae_{2}}
&\D(e_{2}Ae_{2}) \ar@/^-1.5pc/[l]_{?\ot^{\bf L}_{k}e_{2}A} \ar@/^1.5pc/[l]^{\RshHom_{k}(Ae_{2}, ?)}
  }, \]
  where $e_{2}Ae_{2}=k[X]/(X^{2})$ with $\deg X=a_{1}\le 0$ and  $A_{e_{2}}=k[Y]$ with $\deg Y=[\za_{1}\zb_{1}]=a_{1}-1<0$.
Since $e_{2}Ae_{2}$ is a silting object in $\per(e_{2}Ae_{2})$, then $e_{2}A=(e_{2}Ae_{2})\ot_{k}^{\bf L}e_{2}A$ is pre-silting in $\per (A)=\D^{\rm b}(A)$.
Since $A_{e_{2}}$ is silting in $\per(A_{e_{2}})$, then by Theorem \ref{Thm:siltreduc}, we know $\D^{\rm b}(A)$ has a silting object.

 If $a_{1}>0$ and $b_{1}>0$, we consider $A^{!}$. Then at least one of $a'_{1}:=|\za_{1}^{\rm op}|+|\zb_{1}^{\rm op}|=2-a_{1}$ and $b_{1}'=|\zb_{1}^{\rm op}|+|\zg_{1}^{\rm op}|=2-b_{1}$ is non-positive since $a_{1}\ge 2$ or $b_{1}\ge 2$ holds. So $\D^{\rm b}(A^{!})$ admits a silting object by above argument. Since $\D^{\rm b}(A) \simeq \D^{\rm b}(A^{!})$ by Proposition \ref{Prop:Extdual}, we have that $\D^{\rm  b}(A)$ has a silting object.
 \end{proof}

 \begin{proof}[Proof of Theorem \ref{Thm:exist1}]
 Thanks to Proposition \ref{Prop:silting-SMC}, we only need to show the existence of silting objects.
  Assume $A$ is derived equivalent to $A^{(n)}$ and $a_{i}\not=1$ or $b_{i}\not=1$ holds for any $1\le i\le n$.
 We apply induction on $n$.
 The assertion is true for $n=1$ by Lemma \ref{Lem:not=1}. Let $e=e_{1}+e_{2}$.  Assume  $n>1$ and the assertion is true for $1, 2, \dots, n-1$. Note that we have a recollement by Proposition \ref{Thm:Fukaya}
  \[\xymatrixcolsep{4pc}\xymatrix{
\D^{\rm b}(A^{(n-1)}) \ar[r]
&\D^{\rm b}(A^{(n)}) \ar@/_1.5pc/[l] \ar@/^1.5pc/[l]^{}
 \ar[r]^{?\ot_{A}^{\bf L}Ae}
&\D^{\rm b}(A^{(1)}) \ar@/^-1.5pc/[l]_{?\ot^{\bf L}_{eAe}eA} \ar@/^1.5pc/[l]
  }, \]
  where $A^{(1)}=eA^{(n)}e$ and $A^{(n-1)}=(A^{(n)})_{e}$ are given by the sub graded quivers of $Q^{(n)}$.
  Thus $\D^{\rm b}(A^{(n)})$ has silting objects by our assumption and Theorem~\ref{Thm:siltreduc}.

If $A$ is not derived equivalent to $A^{(n)}$, then $\D^{\rm b}(A)$ has a full exceptional sequence by Proposition \ref{Prop:excep}. This exceptional sequence give rises to a silting object in $\D^{\rm b}(A)$ by \cite[Proposition 3.5]{AI}. So the assertion is true.
 \end{proof}

For $A=A^{(1)}$ with $a_{1}=1=b_{1}$, we have  checked, in a case by case analysis, that for every indecomposable object $X$ in $\D^{\rm b}(A)$, we have that  $\Hom(X, X[n])$ is non-zero, for some positive $n$. This shows that the converse of Theorem \ref{Thm:exist1} is also true in this case and that $\per(A^{(1))}$ does not have a silting object.

\

\subsection*{Funding}
The first author acknowledges  support by Shaanxi Normal University and the NSF of China (Grant No. 12271321).
The third author acknowledges  support by the EPSRC Early Career Fellowship EP/P016294/1. The second and third author acknowledges  partial support by  the DFG through the project SFB/TRR 191 Symplectic Structures
in Geometry, Algebra and Dynamics (Projektnummer 281071066--TRR 191).


\end{document}